\documentclass[12pt]{amsart}
\usepackage{amsmath,amssymb,amsfonts,amsthm,amsopn}
\usepackage{latexsym,graphicx}
\usepackage{xcolor}
\usepackage{color, colortbl}
\usepackage{pb-diagram}
\usepackage[title]{appendix}

\newtheorem{theorem}{Theorem}[section]
\newtheorem{lemma}[theorem]{Lemma}
\newtheorem{corollary}[theorem]{Corollary}
\newtheorem{proposition}[theorem]{Proposition}
\newtheorem{definition}[theorem]{Definition}

\newtheorem{example}[theorem]{Example}
\newtheorem{remark}[theorem]{Remark}

\newcommand{\beqa}{\begin{eqnarray*}}
\newcommand{\eeqa}{\end{eqnarray*}}

\newcommand{\field}[1]{\mathbb{#1}}
\newcommand{\bR}{\field{R}}        
\newcommand{\bZ}{\field{Z}}        
\newcommand{\bC}{\field{C}}        
\def\la{\lambda}

\def\cF{\mathcal{F}}              
\def\cS{\mathcal{S}}

\def\cB{\mathcal{B}}

\def\cA{\mathcal{A}}

\def\cI{\mathcal{I}}

\def\rd{\bR^d}

\def\rdd{{\bR^{2d}}}

\def\lrd{L^2(\rd)}
\def\lrdd{L^2(\rdd)}

\def\intrd{\int_{\rd}}
\def\intrdd{\int_{\rdd}}

\def\R{\right)}

\def\<{\left<}
\def\>{\right>}

\def\mv1{M_v^1}

\def\mn{(m,n)}
\def\mn'{(m',n')}

\newcommand{\norm}[1]{\lVert#1\rVert}

\def\R{\mathbb{R}}
\def\Ren{\mathbb{R}^d}

\def\Qs{{Q_s}}

\def\f{\varphi}

\def\Sn2{S_{2}(L^{2}(\Ren))}
\def\S1{S_{1}(L^{2}(\Ren))}
\def\sig00{\sigma_{0,0}}

\def\la{\langle}
\def\ra{\rangle}

\begin{document}

\begin{abstract} 
	In this work, we extend Wigner's original framework to analyze linear operators by examining the relationship between their Wigner and Schwartz kernels. Our approach includes the introduction of (quasi-)algebras of Fourier integral operators (FIOs), which encompass FIOs of type I and II. The symbols of these operators reside in (weighted) modulation spaces, particularly in Sj{\"o}strand's class, known for its favorable properties in time-frequency analysis. One of the significant results of our study is demonstrating the inverse-closedness of these symbol classes.
	
	  Our analysis includes fundamental examples such as pseudodifferential operators and Fourier integral operators related to Schr{\"o}dinger-type equations. These examples typically feature classical Hamiltonian flows governed by linear symplectic transformations $S \in Sp(d, \mathbb{R})$. The core idea of our approach is to utilize the Wigner kernel to transform a Fourier integral operator $ T $ on $ \mathbb{R}^d $ into a pseudodifferential operator $ K$ on $ \mathbb{R}^{2d}$. This transformation involves a symbol $\sigma$  well-localized around the manifold defined by $ z = S w $.
\end{abstract}

\title{Understanding of linear operators through Wigner analysis}

\author{Elena Cordero}
\address{Universit\`a di Torino, Dipartimento di Matematica, via Carlo Alberto 10, 10123 Torino, Italy}
\email{elena.cordero@unito.it}

\author{Gianluca Giacchi}
\address{Universit\'a di Bologna, Dipartimento di Matematica, Piazza di Porta San Donato 5, 40126 Bologna, Italy; University of Lausanne, Switzerland; HES-SO School of Engineering, Rue De L'Industrie 21, Sion, Switzerland; Centre Hospitalier Universitaire Vaudois, Switzerland}
\email{gianluca.giacchi2@unibo.it}

\author{Edoardo Pucci}
\address{Universit\`a di Torino, Dipartimento di Matematica, via Carlo Alberto 10, 10123 Torino, Italy}
\email{edoardo.pucci@edu.unito.it}

\thanks{}
\subjclass{Primary 35S30; Secondary 47G30}

\subjclass[2010]{35S05,35S30,
47G30}
\keywords{Wigner transform, Schr\"odinger equation, , Fourier transform, symplectic group, metaplectic operators}
\maketitle

\section{Introduction}
The focus of this study is the  analysis of linear operators using Wigner distributions, which are a way of representing functions on phase space. 
The original idea stems from Wigner's paper in 1932 \cite{Wigner} who first introduced the Wigner distribution and applied it to represent the phase-space concentration of 
Schr\"{o}dinger type propagators.
Namely, given two signals $f,g\in\lrd$ the (cross-)Wigner distribution is defined by
\begin{equation}\label{WigDist}
	W(f,g)(x,\xi)=\int_{\rd}f(x+t/2)\overline{g(x-t/2)}e^{-2\pi i\xi t}dt, \qquad x,\xi\in\rd. 
\end{equation}
If 	$f=g$  we simply write $Wf= W(f,f)$ and call it the Wigner distribution of $f$.

By the Schwartz kernel theorem (see e.g. \cite{Rudin}) any continuous linear operator $T:\,\cS(\rd)\to \cS'(\rd)$ admits a unique distribution  $k_T\in\cS'(\rdd)$,
 called the Schwartz kernel of  $T$, such that for any test function $f\in\cS(\rd)$,
\begin{equation}
	Tf(x)=\intrd k_T(x,y) f(y)\,dy.
\end{equation}

The density of the subspace $\mbox{span}\{W(f,g):f,g\in\cS(\rd)\}$ in $\cS(\rdd)$ (cf. \cite[Lemma 1.1]{AEFN06}) allows to introduce the Wigner kernel of an operator $T$ as above. 
\begin{definition}\label{1.1}
Given a continuous linear operator $T:\,\cS(\rd)\to \cS'(\rd)$, we define the  linear operator $K:\cS(\rdd)\to \cS'(\rdd)$  by
\begin{equation}\label{I3}
	KW(f,g)=W(Tf,Tg),  \qquad f,g\in\cS(\rd).
\end{equation}
Its  Schwartz kernel $k$ is called the \emph{Wigner kernel}  of $T$:
\begin{equation}\label{I4}
	KW(f,g)(z) = \intrdd k(z,w) W(f,g)(w)\,dw,\quad z\in\rdd,\quad f,g\in\cS(\rd).
\end{equation}
\end{definition}
Such an operator $K$ (resp. Wigner kernel $k$) exists and is unique, as proved in Theorem 3.3 of \cite{CRGFIO1}: 

 \begin{theorem}\label{3.3}
For any continuous linear  $T:\cS(\rd)\to\cS'(\rd)$ with Schwartz  kernel $k_T\in\cS'(\rdd)$, there exists a unique linear operator $K:\cS(\rdd)\to \cS'(\rdd)$ with Schwartz kernel $k\in\cS'(\bR^{4d})$ such that \eqref{I3} holds. Therefore, every continuous linear operator $T:\cS(\rd)\to\cS'(\rd)$ has a unique Wigner kernel. Besides,
	\begin{equation}\label{nuclei}
		k=\mathfrak{T}_pWk_T,
	\end{equation}
	where $\mathfrak{T}_pF(x,\xi,y,\eta)=F(x,y,\xi,-\eta)$.
\end{theorem}

This definition encompasses the original idea of Wigner \cite{Wigner} of considering an operator $K$ on $\cS(\bR^{2d})$ satisfying the relation:
\begin{equation}\label{WignersWkernel}
	W(Tf)=K(Wf), \qquad f\in\cS(\rd),
\end{equation}
which reads, in its integral form, as:
\begin{equation}\label{WignersWkernelIntegral}
	W(Tf)(z)=\int_{\rdd}k(z,w)Wf(w)dw, \qquad f\in\cS(\rd).
\end{equation}
This is the case $f=g$ in  Definition \ref{1.1}.

Our main interest resides in using the Wigner kernel to extract the time-frequency content of operators $T$ which represents propagators for the Schr\"{o}dinger equations.  This idea, originated by  Wigner \cite{Wigner}, was followed by Kirkwood \cite{Kirkwood}, and Moyal \cite{Moyal}.
Later, Cohen and Galleani \cite{CG2,CG1} applied the Wigner distribution 
to classical systems including  acoustics, speech processing, machine monitoring, biomedical signals etc.,  when the governing equation of the variable is a linear ordinary or partial differential equation. They highlighted that many of the methods
that have been developed for the Wigner distribution
in regard to the Schr\"{o}dinger equation should be applicable to other wave equations.
In the same spirit, Mele and Oliaro \cite{MO} studied  regularity of partial differential equations with polynomial coefficients  via Wigner distribution  and proved regularity properties for these classes.

This line of investigation was started by two of us in the works \cite{CGRPartII2022,CRGFIO1}.  Here we continue this research: our scope is to build up  classes of linear operators which can be successfully analysed via Wigner distributions. The core of this study will be the time-frequency concentration of the Wigner kernel of operators.
To measure such concentration we will use modulation spaces.\par
Given a weight function $m$ on the phase space $\rdd$, $0<p,q\leq\infty$, the modulation space $M^{p,q}_m(\rd)$ was first introduced  by H. Feichtinger in the 80's \cite{F1} for the Banach case and by Galperin and Samarah \cite{Galperin2004} and Kobayashi \cite{Kob2005} for the quasi-Banach setting. These spaces were originally defined in terms of the short-time Fourier transform (STFT). Namely, fix a window function $g\in\cS(\rd)\setminus\{0\}$, the STFT of a function/tempered distribution $f$ is defined by
\begin{equation}\label{STFTp}
	V_gf(x,\xi)=\int_{\rd}f(t)\bar g(t-x)e^{-2\pi i\xi\cdot t}dt, \qquad  \ x,\xi\in\rd. 
\end{equation}
A tempered distribution $f\in\cS'(\rd)$ belongs to $M^{p,q}_m(\rd)$ if and only if
$V_gf$ is in the weighted mixed norm space $L^{p,q}_m(\rdd)$ (see the next section for details). Moreover, their norms are equivalent:
\[
\|f\|_{ M^{p,q}_m} \asymp\| V_gf\|_{L^{p,q}_m}. 
\]

In this paper we construct and investigate Wiener subalgebras consisting of Fourier integral operators which generalize  the classes introduced in \cite{CRGFIO1} from symbols in the H\"ormander class $S^0_{0,0}(\rdd)$  to symbols in modulation spaces (including the  Sj\"ostrand Class \cite{A76}). 
Basic examples of such operators are Schr\"odinger propagators $e^{itH}$ with
Hamiltonians\begin{equation}\label{Elena0}
	H=a(x,D)+\sigma(x,D),
\end{equation}
where $a(x,D)$ is the quantization of a quadratic form and $\sigma(x,D)$ is a pseudodifferential operator with a rough symbol $\sigma\in M^{\infty,q}_{1\otimes v_s}(\rdd)$.
The propagators above are special instances of  FIOs of type I and II. Namely, we will study FIOs of type I
\begin{equation}\label{FIO1}
	T_I f(x)=T_{I,\Phi,\sigma}f(x)=\int_{\rd} e^{2\pi i
		\Phi(x,\eta)}\sigma(x,\eta)\widehat{f}(\eta)\,d\eta \, ,
\end{equation}
where $\sigma\in M^{\infty,q}_{1\otimes v_s}(\rdd)$ and  the
real-valued phase $\Phi$ is a quadratic form. 
When $\Phi(x,\eta)=x\eta$ we recapture the
pseudodifferential operators in the Kohn-Nirenberg form
	\begin{equation}\label{PseudoDef}
	\sigma(x,D)f(x)=\int_{\rd}\sigma(x,\eta)\hat f(\eta)e^{2\pi ix\eta}d\eta, \qquad f\in\cS(\rd)
\end{equation}
(see Sections $2$ and $5$ for details).

The main insight of this work consists of 
downgrading FIOs as above  to pseudodifferential operators: we use the Wigner kernel to transform FIOs on $\mathbb{R}^d$ to pseudodifferential operators on $\mathbb{R}^{2d}$ with symbols well localized around certain manifolds.
 
For propagators $e^{itH}$ with $H$ as in \eqref{Elena0}, the related Hamiltonian flow is a linear symplectic map $S_t$, for every $t\in\bR$. Inspired by the Schr\"{o}dinger  case we give the following general definition.

\begin{definition}\label{defFIOclass}
	Consider $S\in \mbox{Sp}(d,\bR)$, $0<q\leq1$, and $s\geq0$. A continuous linear operator $T:\cS(\rd)\to\cS'(\rd)$ belongs to the class $FIO(S,M^{\infty,q}_{1\otimes v_s})$ if its Wigner kernel satisfies:
	\begin{equation}
		k(z,w)=h(z,Sw),
	\end{equation}
	where $h$ is the Schwartz kernel of a pseudodifferential operator (with respect to Kohn-Nirenberg quantization) $\sigma(z,D)$ with $\sigma\in M^{\infty,q}_{1\otimes v_s}(\bR^{4d})$.
\end{definition}
This definition exhibits a certain localization along the manifold $z=Sw$, measured in terms of modulation spaces $M^{\infty,q}_{1\otimes v_s}$.

 Here we will study these classes' properties.
\begin{theorem} \label{tc1old} Consider $0<q\leq1$, $s\geq0$.\par
	
	(i) An operator $T\in FIO(S,M^{\infty,q}_{1\otimes v_s})$ is bounded on $\lrd$.
	
	(ii)  If $T_i\in FIO(S_i,M^{\infty,q}_{1\otimes v_s})$, $i=1,2$, then $T_1 T_2\in
	FIO(S_1 S_2,M^{\infty,q}_{1\otimes v_s})$.
	
	(iii)  If $T\in FIO(S,M^{\infty,q}_{1\otimes v_s})$ is invertible on $L^2(\rd)$,
	then $T^{-1}\in FIO(S^{-1},M^{\infty,q}_{1\otimes v_s})$.
\end{theorem}

We will prove that FIOs of type I and II as above fall into this setting.
Key tool will be the relation between Schwartz and Wigner kernel in terms of time-frequency concentration:
	\begin{equation}\label{th23}
	\norm{k}_{M^p_{m}}\asymp\norm{k_T}^2_{M^p_{m}},\quad 0<p\leq\infty, 
\end{equation}
cf. Proposition \ref{prop34} below.
Finally, Theorem \ref{teofinal} will exhibit that Schr\"{o}dinger propagators $e^{itH}$ satisfy 
$$e^{itH}\in  FIO(S_t,M^{\infty,q}_{1\otimes v_s}),\quad 0<q\leq1,$$ where, for every fixed $t\in\bR$, the symplectic matrix $S_t$ describes the classical Hamiltonian flow related to $H$. This result extends a similar one obtained in \cite{CGRPartII2022}, where the perturbation $\sigma(x,D)$ has symbol $\sigma$ in  the  H\"ormander class
$ S^0_{0,0}(\rdd)\subset M^{\infty,q}_{1\otimes v_s}(\rdd)$. Our  conjecture is that other propagators, solutions to dispersive equations, should fall in these classes.\par
The organization of this paper is as follows. We start with a preliminary
section devoted to the definition and basic properties of symplectic matrices, metaplectic operators and $\cA$-Wigner distributions; we recall modulation spaces and present the continuity properties of metaplectic operators and $\cA$-Wigner distributions  on them. We will recall the algebra and Wiener properties for pseudodifferential operators. Section $3$ is devoted to the study of the Wigner kernel, in particular we shall prove \eqref{th23}. Section $4$ studies the properties of  $FIO(S,M^{\infty,q}_{1\otimes v_s})$ and proves Theorem \ref{tc1old}. The last section shows that the FIOs of Type I and II introduced above fall in the class $FIO(S,M^{\infty,q}_{1\otimes v_s})$. This inclusion will be the key tool in proving that  Schr\"{o}dinger propagators are members of these classes:  
$e^{itH}\in  FIO(S_t,M^{\infty,q}_{1\otimes v_s}),$ for every $t\in\bR$.

\section{Preliminaries}\label{sec:2}
	We denote by $x\xi=x\cdot \xi$, $x,\xi\in\rd$, the standard inner product on $\rd$. The space $\cS(\rd)$ is the Schwartz space of rapidly decreasing functions and $\cS'(\rd)$ is its topological dual, the space of tempered distributions. $\la f,g\ra=\int_{\rd}\overline{f(x)}g(x)dx$ denotes the sesquilinear inner product of the functions $f,g\in L^2(\rd)$. If $f\in\cS'(\rd)$ and $g\in\cS(\rd)$, the same notation stands for the unique extension of the $L^2$ inner product to a duality pairing (antilinear in the second component) on $\cS'\times\cS$. $B(L^2(\rd))$ is the space of bounded linear operators on $L^2(\rd)$, the norm of $T\in B(L^2(\rd))$ is denoted by $\norm{T}_{op}$. $M(d,\bR)$ is the group of real $d\times d$ matrices.  $\mbox{Sym}(d,\bR)$ denotes the group of symmetric $d\times d$ matrices, i.e., $C\in\mbox{Sym}(d,\bR)$ if $C^T=C$, whereas $\mbox{GL}(d,\bR)$ denotes the group of $d\times d$ invertible matrices. For $f,g\in L^2(\rd)$, $f\otimes g(x,y)=f(x)g(y)$. This definition extends to $f,g\in\cS'(\rd)$, where $f\otimes g\in\cS'(\rdd)$ is the unique distribution satisfying $\la f\otimes g,\f\otimes\psi\ra=\la f,\f\ra\la g,\psi\ra$.

	\subsection{Symplectic group}
	In this paper, we will make considerable use of the properties of the symplectic group. Let
		\begin{equation}\label{defJ}
			J=\begin{pmatrix}
			0_{d\times d} & I_{d\times d}\\
			-I_{d\times d} & 0_{d\times d}
			\end{pmatrix},
		\end{equation}
		where $0_{d\times d}$ is the $d\times d$ matrix having all zero entries and $I_{d\times d}$ is the $d\times d$ identity matrix. A matrix $S\in\bR^{2d\times 2d}$ is \emph{symplectic} if $S^TJS=J$. We denote the group of $2d\times2d$ symplectic matrices by $\mbox{Sp}(d,\bR)$. 
		
		Equivalently, writing:
		\begin{equation}\label{decompS}
			S=\begin{pmatrix}
				A & B\\ C & D
			\end{pmatrix}, \qquad A,B,C,D\in\bR^{d\times d},
		\end{equation}
		$S\in \mbox{Sp}(d,\bR)$ if and only if:
		\begin{equation}\label{symprel}
			\begin{cases}
				A^TC=C^TA,\\
				B^TD=D^TB,\\
				A^TD-C^TB=I_{d\times d}.
			\end{cases}
		\end{equation}
		We say that $S$ is \emph{upper block triangular} if $C=0_{d\times d}$, \emph{lower block triangular} if $B=0_{d\times d}$ and \emph{block diagonal} if $B=C=0_{d\times d}$.
		
		The matrices:
		\begin{equation}\label{VCDL}
			V_C = \begin{pmatrix}
				I_{d\times d} & 0_{d\times d}\\
				C & I_{d\times d}
			\end{pmatrix}, \quad and \quad \mathcal{D}_L=\begin{pmatrix} L^{-1} & 0_{d\times d} \\ 0_{d\times d} & L^T \end{pmatrix},
		\end{equation}
		for $C\in\mbox{Sym}(d,\bR)$ and $L\in \mbox{GL}(d,\bR)$, are symplectic. 
		
		\begin{proposition}\label{genSp} 
		$\mbox{Sp}(d,\bR)$ is generated by:
		\[
			\{V_C : C\in\mbox{Sym}(d,\bR)\}\cup\{\mathcal{D}_L : L\in\mbox{GL}(d,\bR)\}\cup\{ J\}.
		\]
		\end{proposition}

	\subsection{Metaplectic operators}
		For $x,\xi\in\rd$, the \emph{translation operator} $T_x$ and the \emph{modulation operator} $M_\xi$ are defined as:
		\begin{equation}
		 T_xg(t)=g(t-x), \quad and \quad M_\xi g(t)=e^{2\pi i\xi t}g(t),
		 \end{equation}
		$g\in L^2(\rd)$. They extend to $g\in\cS'(\rd)$ by duality:
		\[
		\la T_x g,\varphi\ra = \la g,T_{-x}\varphi\ra, \qquad \la M_\xi g,\f\ra = \la g,M_{-\xi}\f\ra, \qquad \f\in\cS(\rd).	
		\]
		The composition of these operators $\pi(x,\xi)=M_\xi T_x$ is called \emph{time-frequency shift}. 
		
		For a given matrix $S\in\mbox{Sp}(d,\bR)$, there exist $\hat S\in B(L^2(\rd))$ unitary and $c_S\in\bC$, $|c_S|=1$, such that
		\begin{equation}\label{defMp}
			\hat S^{-1}\pi(x,\xi)\hat S=c_S\pi(S(x,\xi)), \qquad x,\xi\in\rd.
		\end{equation}
		The operator $\hat S$ is not unique, but if $\hat S'$ is another operator satisfying \eqref{defMp}, then $\hat S'=c\hat S$, for some $c\in\bC$ such that $|c|=1$. The set $\{\hat S:S\in\mbox{S}(d,\bR)\}$ is a group under composition, and it admits a subgroup containing exactly two operators for each $S\in \mbox{Sp}(d,\bR)$, differing by a sign. This subgroup is denoted by $\mbox{Mp}(d,\bR)$, the \emph{metaplectic group}. 
		
		The projection $\pi^{Mp}:\mbox{Mp}(d,\bR)\to\mbox{Sp}(d,\bR)$, defined by $\pi^{Mp}(\hat S)=S$ is a group homomorphism with $\ker(\pi^{Mp})=\{\pm Id_{L^2}\}$. Throughout this work, if $\hat S\in \mbox{Mp}(d,\bR)$, $S$ (without caret) denotes its (unique) symplectic projection. If $\hat S,\hat S_1,\hat S_2\in \mbox{Mp}(d,\bR)$, the following identities hold up-to-a-sign:
		\begin{equation}\label{homomorph}
			\hat S_1\circ \hat S_2=\widehat{S_1S_2}, \quad and \quad \hat S^{-1}=\widehat{S^{-1}}.
		\end{equation}
		Metaplectic operators enjoy the following continuity properties.
		\begin{proposition}
		Let $\hat S\in\mbox{Mp}(d,\bR)$.\\
			(i) $\hat S:L^2(\rd)\to L^2(\rd)$ is unitary.\\
			(ii) The restriction of $\hat S$ to $\cS(\rd)$ is a homeomorphism of $\cS(\rd)$.\\
			(iii) $\hat S$ extends to a homeomorphism of $\cS'(\rd)$ as:
			\[
				\la \hat Sf,g\ra = \la f,\hat S^{-1}g\ra, \qquad f\in\cS'(\rd), \ g\in\cS(\rd).
			\]
		\end{proposition}
		\begin{example}\label{ex23}
		Recall the definitions of the matrices $J$, given in \eqref{defJ}, $\mathcal{D}_L$ and $V_C$, given in \eqref{VCDL}. \\
		(i) The \emph{Fourier transform} $\cF$ is a metaplectic operator. Recall that the Fourier transform of $f\in \cS(\rd)$ is defined as:
		\[
			\cF f(\xi)=\hat f(\xi)=\int_{\rd}f(x)e^{-2\pi i\xi x}dx, \qquad \xi\in\rd.
		\]
		We have: $\pi^{Mp}(\cF)=J$.\\
		(ii) \emph{Unitary linear changes of variables} are metaplectic operators. Namely, if $L\in \mbox{GL}(d,\bR)$, the operator 
		\begin{equation}\label{LchVar}
			\mathfrak{T}_Lf(t)=|\det L|^{1/2}f(Lt), \qquad f\in L^2(\rd),
		\end{equation}
		is a metaplectic operator. We have: $\pi^{Mp}(\mathfrak{T}_L)=\mathcal{D}_L$. 
		In particular, we set:
		\begin{equation}\label{TW}
			\mathfrak{T}_{w}=\mathfrak{T}_{L_{1/2}},
		\end{equation}
		where
		\[
			L_{1/2}=\begin{pmatrix}I_{d\times d} & I_{d\times d}/2\\
			I_{d\times d} & -I_{d\times d}/2\end{pmatrix}.
		\]
		(iii) For $C\in\mbox{Sym}(d,\bR)$, we consider the \emph{chirp function}: $\Phi_C(t)=e^{i\pi Ct\cdot t}$. \\
		(iiia) The operator:
		\[
			\phi_C f(t)=\Phi_C(t)f(t), \qquad f\in L^2(\rd)
		\]
		is metaplectic, with $\pi^{Mp}(\phi_C)=V_C$. \\
		(iiib) The Fourier multiplier:
		\[
			\psi_Cf(t)=\widehat{\Phi_{-C}}\ast f(t)=(\cF^{-1} \Phi_{-C} \cF f)(t), \qquad f\in L^2(\rd),
		\]
		is metaplectic, with $\pi^{Mp}(\psi_C)=V_C^T$.
		\end{example}
		\begin{example}
		We will also consider the \emph{partial Fourier transform} with respect to the frequency variables:
		\begin{equation}\label{FT2}
			\cF_2F(x,\xi)=\int_{\rd}F(x,t)e^{-2\pi i\xi t}dt, \qquad x,\xi\in\rd,
		\end{equation}		
		which is a metaplectic operator in $\mbox{Mp}(2d,\bR)$, with projection \cite{CGRPartII2022}:
		\begin{equation}\label{AFT2}
			\pi^{Mp}(\cF_2)=\cA_{FT2}=\begin{pmatrix}
				I_{d\times d} & 0_{d\times d} & 0_{d\times d} & 0_{d\times d}\\
				0_{d\times d} & 0_{d\times d} & 0_{d\times d} & I_{d\times d}\\
				0_{d\times d} & 0_{d\times d} & I_{d\times d} & 0_{d\times d}\\
				0_{d\times d} & -I_{d\times d} & 0_{d\times d} & 0_{d\times d}
			\end{pmatrix}.
		\end{equation}
		\end{example}
		
		The following lifting property is proved in \cite[Theorem B.1]{CG2023}.
		\begin{theorem}\label{lift}
			Let $\hat S_1,\hat S_2\in \mbox{Mp}(d,\bR)$. There exists a unique $\hat S\in \mbox{Mp}(2d,\bR)$ characterized by $\hat S(f\otimes g)=\hat S_1f\otimes\hat S_2g$. If
			\[
				S_1=\begin{pmatrix}
				A_1 & B_1\\
				C_1 & D_1
				\end{pmatrix} \qquad and \qquad 
				S_2=\begin{pmatrix}
					A_2 & B_2\\
					C_2 & D_2
				\end{pmatrix},
			\]
			then
			\[
				S=\begin{pmatrix}
				A_1 & 0_{d\times d} & B_1 & 0_{d\times d}\\
				 0_{d\times d} & A_2 &  0_{d\times d} & B_2\\
				C_1 & 0_{d\times d} & D_1 &  0_{d\times d} \\
				 0_{d\times d} & C_2 &  0_{d\times d} & D_2
				\end{pmatrix}.
			\]
		\end{theorem}
		
		Moreover, we mention the following intertwining relation, which is a particular case of \cite[Proposition A.1]{CRGFIO1}.
		
		\begin{proposition}\label{intert}
			Let $M\in\mbox{GL}(2d,\bR)$ and $\hat U:=\cF_2^{-1}\mathfrak{T}_M\cF_2$. Then, $U=\pi^{Mp}(\hat U)$ is upper block triangular if and only if $M$ is upper block triangular. Moreover, $\hat U=\mathfrak{T}_N$ for some $N\in \mbox{GL}(2d,\bR)$ if and only if $M$ is block diagonal.
		\end{proposition}
		
		\subsection{Metaplectic Wigner distributions}
		Metaplectic Wigner distributions are defined in \cite{CRPartI2022}, and later studied in \cite{CG2023,CGRPartII2022}, see also the recent contribution \cite{CharlyIrina2024}.
		\begin{definition}
			Let $\hat\cA\in \mbox{Mp}(2d,\bR)$. The (cross-)\emph{metaplectic Wigner distribution} associated to $\hat\cA$ is the time-frequency representation
			\begin{equation}\label{MWD}
				W_\cA(f,g)=\hat\cA(f\otimes\bar g), \qquad f,g\in\cS'(\rd).
			\end{equation}
		\end{definition}
		We write $W_\cA f=W_\cA(f,f)$. Metaplectic Wigner distributions satisfy the following continuity properties.
		\begin{proposition}\label{propcont}
			(i) $W_\cA:L^2(\rd)\times L^2(\rd)\to L^2(\rdd)$ is bounded, with
			\begin{equation}\label{moyal}
				\la W_\cA(f,g),W_\cA(u,v)\ra =\la f,u\ra \overline{\la g,v\ra}, \qquad f,g,u,v\in L^2(\rd)
			\end{equation}
			(Moyal's identity). In particular, $\norm{W_\cA f}_2=\norm{f}_2^2$.\\
			(ii) $W_\cA:\cS(\rd)\times\cS(\rd)\to\cS(\rdd)$ is continuous.\\
			(iii) $W_\cA:\cS'(\rd)\times\cS'(\rd)\to\cS'(\rdd)$ is continuous.
		\end{proposition}
		We will consider a metaplectic operator $\hat\cA\in \mbox{Mp}(2d,\bR)$ and its projection, i.e. the $4d\times 4d$ symplectic matrix $\cA\in \mbox{Sp}(2d,\bR)$ with $d\times d$ block decomposition:
		\begin{equation}\label{blocksA}
			\cA = \begin{pmatrix}
				A_{11} & A_{12} & A_{13} & A_{14}\\
				A_{21} & A_{22} & A_{23} & A_{24}\\
				A_{31} & A_{32} & A_{33} & A_{34}\\
				A_{41}& A_{42} & A_{43} & A_{44}
			\end{pmatrix}.
		\end{equation}
		The submatrix
		\begin{equation}\label{defEA}
			E_\cA = \begin{pmatrix}
				A_{11} & A_{13}\\
				A_{21} & A_{23}
			\end{pmatrix}
		\end{equation}
		plays a special role, that will be clarified in Theorem \ref{thm26} below.
		\begin{definition}\label{shiftinv}
		Under the notation above, a metaplectic Wigner distribution $W_\cA$ (or equivalently, the projection $\cA\in\mbox{Sp}(2d,\bR)$) is \emph{shift-invertible} if $E_\cA\in \mbox{GL}(2d,\bR)$.
		\end{definition}
		
		\begin{example}\label{exW0}
		Examples of metaplectic Wigner distributions can be found in \cite{CG2023,CGRPartII2022}. In the present work, we limit to mention two important intances of these time-frequency representations: the cross-Wigner distribution, described in Subsection \ref{subsec:Wigner}, and the (cross-)Rihacek distribution, which is defined as
		\begin{equation}
			W_0(f,g)(x,\xi)=f(x)\overline{\hat g(\xi)}e^{-2\pi i\xi x}, \qquad f,g\in\cS(\rd).
		\end{equation}
		 It is easy to verify that $$W_0(f,g)=\cF_2\mathfrak{T}_{L_0}(f\otimes\bar g),$$ where $\mathfrak{T}_{L_0}F(x,\xi)=F(L_0(x,\xi))=F(x,x-\xi)$ and $\cF_2$ is defined in \eqref{FT2}. Also, using \eqref{homomorph},  $\cF_2\mathfrak{T}_{L_0}=\hat A_0$, where
		\begin{equation}\label{defA0}
			A_0=\begin{pmatrix}
				I_{d\times d} & 0_{d\times d} & 0_{d\times d} & 0_{d\times d}\\
				0_{d\times d} & 0_{d\times d} & 0_{d\times d} & -I_{d\times d}\\
				0_{d\times d} & 0_{d\times d} & I_{d\times d} & I_{d\times d}\\
				-I_{d\times d} & I_{d\times d} & 0_{d\times d} & 0_{d\times d}
			\end{pmatrix}.
		\end{equation}
		\end{example}
		
	\subsection{The Wigner distribution}\label{subsec:Wigner}
		The \emph{(cross-)Wigner distribution} of $f\in L^2(\rd)$ with respect to the window $g\in L^2(\rd)$ is defined as in \eqref{WigDist}. This definition can be extended to $f\in\cS'(\rd)$ and $g\in\cS(\rd)$ as follows: 
		\begin{equation}\label{defW2}
			W(f,g)(x,\xi)=\la f,\pi_w(x,\xi)g\ra, \qquad x,\xi\in\rd,
		\end{equation}
		where the \emph{metaplectic atoms} $\pi_w(x,\xi)$ are defined as:
		\begin{equation}
			\pi_w(x,\xi)g(t)=2^de^{4\pi i(t- x)\xi}g(2x-t), \qquad x,\xi\in\rd.
		\end{equation}
		Finally, the definition of the cross-Wigner distribution can be extended to $f,g\in\cS'(\rd)$ by setting:
		\begin{equation}\label{MWW}
			W(f,g)=\cF_2\mathfrak{T}_w(f\otimes\bar g),
		\end{equation}
		where $\cF_2$ is defined as in \eqref{FT2} and $\mathfrak{T}_w$ is defined as in \eqref{TW}. The \emph{Wigner distribution} of $f\in \cS'(\rd)$ is $Wf=W(f,f)$. 
		
		\begin{remark}
			Formula \eqref{MWW} highlights the fact that the Wigner distribution is the (shift-invertible) metaplectic Wigner distribution associated to $\hat A_{1/2}=\cF_2\mathfrak{T}_w$. The projection $A_{1/2}$ has block decomposition:
			\begin{equation}\label{defA12}
			A_{1/2} = \begin{pmatrix}
				I_{d\times d}/2 & I_{d\times d}/2 & 0_{d\times d} & 0_{d\times d}\\
				0_{d\times d} & 0_{d\times d} & I_{d\times d}/2 & -I_{d\times d}/2\\
				0_{d\times d} & 0_{d\times d} & I_{d\times d} & I_{d\times d}\\
				-I_{d\times d} & I_{d\times d} & 0_{d\times d} & 0_{d\times d}
			\end{pmatrix}.
			\end{equation}
		\end{remark}

		\subsection{Modulation spaces}
		A \emph{submultiplicative weight} on $\rdd$ is a function $v:\rdd\to [0,+\infty)$ such that $v(z+w)\leq v(z)v(w)$ for every $z,w\in\rdd$. If $v$ is a submultiplicative weight, a function $m:\rdd\to[0,+\infty)$ is a $v$-\emph{moderate} weight if $m(z+w)\lesssim m(z)v(w)$. We denote with $\mathcal{M}_v(\rdd)$ the set of $v$-moderate weights. 
		
		We will use the polynomial weight functions: $v_s(z)=(1+|z|^2)^{s/2}$, for $s\geq0$. These are submultiplicative weights satisfying:
		\[
			v_s(x,\xi)\leq v_s(x)v_s(\xi).
		\] 
		Let $0<p,q\leq\infty$, $m\in\mathcal{M}_v(\rdd)$ and $g\in\cS(\rd)\setminus\{0\}$. For $f\in\cS'(\rd)$, consider the quasi-norm:
			\[
				\norm{f}_{M^{p,q}_m}=\left(\int_{\rd}\left(\int_{\rd} |W(f,g)(x,\xi)|^pm(x,\xi)^pdx\right)^{q/p}d\xi\right)^{1/q},
			\]
			with obvious modifications if either $p=\infty$ or $q=\infty$. The definition of $\norm{\cdot}_{M^{p,q}_m}$ does not depend on the choice of $g$, i.e., different windows yield to equivalent quasi-norms. 
			
			The \emph{modulation space} $M^{p,q}_m(\rd)$ is the space of tempered distributions $f\in \cS'(\rd)$ such that $\norm{f}_{M^{p,q}_m}<\infty$. If $p=q$, $M^p_m(\rd):=M^{p,p}_m(\rd)$ and we write $M^{p,q}(\rd)$ if $m=1$. Modulation spaces enjoy the following inclusion relations: for $0<p_1\leq p_2\leq\infty$, $0<q_1\leq q_2\leq\infty$ and $m_1,m_2\in\mathcal{M}_v(\rdd)$ satisfy $m_2\lesssim m_1$,
			\[
				\cS(\rd)\hookrightarrow M^{p_1,q_1}_{m_1}(\rd)\hookrightarrow M^{p_2,q_2}_{m_2}(\rd)\hookrightarrow\cS'(\rd),
			\]
			the inclusions being continuous. If $p_1,p_2,q_1,q_2\neq\infty$, the inclusions are also dense. 
			
			When defining modulation spaces, the Wigner distribution can be replaced by any {shift-invertible} {metaplectic Wigner distribution}, as detailed in \cite{CG2023}. 
			
			\begin{theorem}\label{thm26}
				Let $W_\cA$ be a shift-invertible metaplectic Wigner distribution and $m\in\mathcal{M}_v(\rdd)$ satisfying $m\circ E_\cA\asymp E_\cA$, where $E_\cA$ is defined as in \eqref{defEA}. Let $g\in\cS(\rd)\setminus\{0\}$. \\
				(i) If $E_\cA$ is upper block triangular, for every $0<p,q\leq\infty$, 
				\begin{equation}\label{charMpq}
				\norm{f}_{M^{p,q}_m}\asymp\norm{W_\cA(f,g)}_{L^{p,q}_m}, \qquad f\in\cS'(\rd).
				\end{equation}
				(ii) If $0<p\leq\infty$, 
				\[
					\norm{f}_{M^{p}_m}\asymp\norm{W_\cA(f,g)}_{L^{p}_m}, \qquad f\in\cS'(\rd).
				\]
			\end{theorem}
			We recall the following characterization \cite{Elena-book}:
			\begin{lemma}\label{lemmaE0}
				If $m\in\mathcal{M}_v(\rdd)$, $1\leq p,q<\infty$ then $\cS(\rd)$ is dense in $M^{p,q}_m(\rd)$ and $M^{p,q}_m(\rd)^*=M^{p',q'}_{1/m}(\rd)$, with $1/p+1/p'=1$, $1/q+1/q'=1$. Moreover, if $1<p,q\leq\infty$,
				\begin{equation}\label{E0}
				\|f\|_{ M^{p,q}_m}=\sup |\la f,g\ra|
				\end{equation}
		where the supremum is taken over all $g\in \cS(\rd)$ such that $\|g\|_{ M^{p',q'}_{1/m}}=1$.
			\end{lemma}
			Metaplectic operators exhibit nice continuity properties on modulation spaces, c.f. \cite{Fuhr, CG2023}.
			\begin{proposition}\label{prop25}
				Let $\hat S\in \mbox{Mp}(d,\bR)$, $m\in\mathcal{M}_v(\rdd)$, $m\circ S\asymp m$.\\
				(i) $\hat S:M^p_m(\rd)\to M^p_m(\rd)$ is bounded (and a homeomorphism) for every $0<p\leq\infty$.\\
				(ii) If $S$ is upper block triangular, $\hat S:M^{p,q}_m(\rd)\to M^{p,q}_m(\rd)$ is bounded (and a homeomorphism) for every $0<p,q\leq\infty$. 
				\end{proposition}
				
			 We will make use of the following lemma.
			
			\begin{lemma}\label{lemma28}
				Let $\varphi(t)=e^{-\pi|t|^2}$ and $\Phi=\f\otimes \f$. Then, for every $f\in \cS'(\rd)$,
				\begin{equation}
					W(f\otimes \bar f,\Phi)(x,y,\xi,\eta)=W(f,\f)(x,\xi)\overline{\cI_2W(f,\f)(y,\eta)}, \qquad x,\xi,y,\eta\in\rd,
				\end{equation}
				where $\cI_2F(x,\xi)=F(x,-\xi)$ is the flip operator in the frequency variables.
			\end{lemma}
			\begin{proof}
			It is a straightforward computation. We write it for sake of clarity.	We use \eqref{defW2}, together with:
				\begin{align*}
					\pi_w(x,y,\xi,\eta)\Phi(u,v)&=2^{2d}e^{4\pi i[(u-x)\xi+(v-y)\eta]}(\f\otimes\f)(2x-u,2y-v)\\
					&=[2^de^{4\pi i(u-x)\xi}\f(2x-u)]\cdot[2^de^{4\pi i(v-y)\eta}\f(2y-v)]\\
					&=\pi_w(x,\xi)\f(u)\pi_w(y,\eta)\f(v)\\
					&=(\pi_w(x,\xi)\f\otimes\pi_w(y,\eta)\f)(u,v).
				\end{align*}
				It follows that:
				\begin{align*}
					W(f\otimes\bar f,\Phi)(x,y,\xi,\eta)&=\la f\otimes\bar f,\pi_w(x,y,\xi,\eta)\Phi\ra
					=\la f\otimes\bar f,\pi_w(x,\xi)\f\otimes\pi(y,\eta)\f\ra\\
					&=\la f,\pi_w(x,\xi)\f\ra\la \bar f,\pi_w(y,\eta)\f\ra
					=\la f,\pi_w(x,\xi)\f\ra\overline{\la f,{\pi_w(y,-\eta)\f}\ra}\\
					&=W(f,\f)(x,\xi)\overline{W(f,\f)(y,-\eta)}
					=W(f,\f)(x,\xi)\overline{\cI_2W(f,\f)(y,\eta)},
				\end{align*}
				where we used that
				\[
					\overline{\pi_w(y,\eta)\f(t)}=2^de^{4\pi i(t-y)(-\eta)}\f(t-y)=\pi_w(y,-\eta)\f(t).
				\]
				This concludes the proof.
			\end{proof}
			
			We will also use the following corollary of Lemma \ref{lemma28}.
			\begin{corollary}\label{cor29}
			Consider $0<p,q\leq\infty$, $s\geq0$, and let $m$ be either $m(w,z)=v_s(w,z)$ or $m(w,z)=(1\otimes v_s)(w,z)$. Then,\\
				(i) If $f\in M^{p,q}_{m}(\rd)$ then $f\otimes \bar f\in M^{p,q}_{m}(\rdd)$ with
				\begin{equation}\label{E1}
				\norm{f\otimes\bar f}_{M^{p,q}_{ m}}\leq\|f\|^2_{M^{p,q}_{ m}}.
				\end{equation}
				(ii) We have the characterization: $f\in M^{p,q}(\rd)$ if and only if $f\otimes\bar f\in M^{p,q}(\rdd)$ with 	\begin{equation}\label{E2}
					\norm{f\otimes\bar f}_{M^{p,q}}=\|f\|^2_{M^{p,q}}.
				\end{equation}
			\end{corollary}
			\begin{proof} The proof is straightforward, we detail it for sake of clarity. Consider  $\f(t)=e^{-\pi|t|^2}$,  $\Phi=\f\otimes\f$ and $m(x,y,\xi,\eta)= v_s(x,y,\xi,\eta)$. By Lemma \ref{lemma28}, and using 
				$v_s(x,y,\xi,\eta)\leq v_s(x,\xi)v_s(y,\eta)$,
				\begin{align*}
					&\norm{f\otimes\bar f}_{M^{p,q}_{ v_s}}=\norm{W(f\otimes\bar f,\Phi)}_{L^{p,q}_{ v_s}}=\norm{W(f,\f)\otimes\overline{\cI_2W(f,\f)}}_{L^{p,q}_{ v_s}}\\
					&\leq\Big(\int_{\rdd}\Big(\int_{\rdd}|W(f,\f)(x,\xi)|^pv_s(x,\xi)^p|W(f,\f)(y,-\eta)|^pv_s(y,\eta)^pdxdy\Big)^{q/p}d\xi d\eta\Big)^{1/q}.
				\end{align*}

				Observe that if $s=0$, the previous inequality is an equality. Let us focus on the inner integral: by Tonelli's theorem,
				\begin{align*}
					\int_{\rdd}|W(f,\f)(x,\xi)|^p&v_s(x,\xi)^p|W(f,\f)(y,-\eta)|^pv_s(y,\eta)^pdxdy\\
					&=\int_{\rd} |W(f,\f)(x,\xi)|^pv_s(x,\xi)^pdx\int_{\rd}|W(f,\f)(y,-\eta)|^pv_s(y,\eta)^pdy\\
					&=\int_{\rd} |W(f,\f)(x,\xi)|^pv_s(x,\xi)^pdx\int_{\rd}|W(f,\f)(y,\eta)|^pv_s(y,\eta)^pdy.
				\end{align*}
				Therefore,			
\begin{align*}
					\norm{f\otimes\bar f}_{M^{p,q}_{1\otimes v_s}}&\leq\norm{W(f,\f)}_{L^{p,q}_{ v_s}}^2=\norm{f}_{M^{p,q}_{ v_s}}^2.
				\end{align*}
				The case  $m(w,z)=(1\otimes v_s)(w,z)$ is proved analogously, using $v_s(\xi,\eta)\leq v_s(\xi)v_s(\eta)$. $(ii)$ follows from $(i)$ and the observation for $s=0$.
			\end{proof}
			
	\subsection{Pseudodifferential operators and Fourier integral operators of type I}
	In this work we limit to consider Fourier integral operators with phases represented by quadratic polynomials, related studies are in \cite{CSW}. Namely, 
		\begin{equation}\label{QuadTamePhase}
		\Phi(x,\eta)=\frac{1}{2}xC^{-1}x+\eta A^{-1}x-\frac{1}{2}\eta A^{-1}B\eta, \qquad x,\eta\in\rd.
	\end{equation}
	The corresponding \emph{canonical transformation} (see \cite{CRGFIO1}) is a linear mapping $S:\rdd\to\rdd$ represented by the symplectic matrix $S$ having block decomposition \eqref{decompS}, with $\det A\neq0$.
	
	Let $\sigma\in\cS'(\rdd)$ and $\Phi$ a phase as above. The \emph{Fourier integral operator} (FIO) of type I with \emph{symbol} $\sigma$ and tame phase $\Phi$ is the linear operator $T_I:\cS(\rd)\to\cS'(\rd)$ given by:
	\begin{equation*}
		T_If(x)=\int_{\rd}\sigma(x,\eta)\hat f(\eta)e^{2\pi i\Phi(x,\eta)}d\eta, \qquad f\in\cS(\rd).
	\end{equation*}
	 the integral above must be interpreted in the weak sense of distributions, i.e., $Tf\in\cS'(\rd)$ is the tempered distribution defined by its action on $g\in\cS(\rd)$ by:
	\[
		\la Tf,g\ra = \la \sigma , \phi W_0(g,f)\ra.
	\]
	where $W_0$ is the (cross-)Rihacek distribution defined in Example \ref{exW0} and $\phi(u,v)=e^{-2\pi i[\Phi(u,v)-uv]}$.
	
	\begin{example}\label{Pseudo} Particular instances of FIOs are \emph{pseudodifferential operators}, which correspond to the choice $\Phi(x,\eta)=x\eta$ in \eqref{FIO1}. Namely, for a fixed $\sigma\in\cS'(\rdd)$, the Kohn-Nirenber operator with symbol $\sigma$ is  $\sigma(x,D):\cS(\rd)\to\cS'(\rd)$ given by \eqref{PseudoDef}.
	\end{example}
	
	The Schwartz kernel of a bounded operator $T:\cS(\rd)\to\cS'(\rd)$ is the (unique) tempered distribution $k_T\in\cS'(\rdd)$ such that:
	\begin{equation}\label{Skernel}
		\la Tf,g\ra = \la k_T,g\otimes \bar f\ra, \qquad f,g\in\cS(\rd).
	\end{equation}
	The Schwartz kernel of a pseudodifferential operator $\sigma(x,D)$ is given by:
	\begin{equation}\label{kerandsymb}
		k_T(x,\xi)=\cF_2\sigma(x,\xi-x), 
	\end{equation}
	where $\cF_2$ is the partial Fourier transform defined in \eqref{FT2}. If $T_1,T_2:\cS(\rd)\to\cS'(\rd)$ are bounded and linear, and $T=T_1T_2:\cS(\rd)\to\cS'(\rd)$ is bounded and linear, the relation between the Schwartz kernels $k_{T_j}$ of $T_j$ ($j=1,2$) and the Schwartz kernel $K_T$ of $T$ is:
	\begin{equation}\label{Skercomp}
		k_T(x,y)=\int_{\rd}k_{T_1}(x,z)k_{T_2}(z,y)dz.
	\end{equation}As customary,

	Any metaplectic Wigner distribution $W_\cA$ can be used as a quantization for pseudodifferential operators.
	\begin{definition}\label{pseudoA}
		Let $W_\cA$ be a metaplectic Wigner distribution. Let $\sigma\in\cS'(\rdd)$. The metaplectic pseudodifferential operator with symbol $\sigma$ and quantization $W_\cA$ is the operator $Op_\cA(\sigma):\cS(\rd)\to\cS'(\rd)$ such that:
		\[
			\la Op_\cA(\sigma)f,g\ra=\la \sigma,W_\cA(g,f)\ra, \qquad f,g\in\cS(\rd).
		\]
	\end{definition}
	
	The Schwartz kernel $k_T$ of a metaplectic pseudodifferential operator $T=Op_\cA(\sigma)$ is related to the symbol $\sigma$ by the relation:
	\begin{equation}\label{genkersymb}
		k_T=\hat\cA^{-1}\sigma.
	\end{equation}
	
	\begin{example}
		(i) The Kohn-Nirenberg quantization of Example \ref{Pseudo} corresponds to Definition \ref{pseudoA} with $W_\cA=W_0$, where $W_0$ is the (cross)-Rihacek distribution of Example \ref{exW0}. Formula \eqref{kerandsymb} is just a restatement of \eqref{genkersymb} for this particular case.
		(ii) The Weyl quantization corresponds to the choice $W_\cA=W$, the (cross-)Wigner distribution. 
	\end{example}
	
In this paper we will use the  notation: $$a^w(x,D):=Op_{A_{1/2}}(a).$$
	For the theory of metaplectic pseudodifferential operators, we refer to \cite{CGRPartII2022}. In this work, we shall use the following change-of-quantization law.
	
	\begin{theorem}\label{changeOfQuant}
		Let $W_\cA$ and $W_\cB$ be metaplectic Wigner distributions. Then, $Op_\cA(a)=Op_\cB(b)$ if and only if $b=\hat\cB\hat\cA^{-1}a$.
	\end{theorem}
	
	Let us recall the following algebra and Wiener property for pseudodifferential operators \cite[Theorems 3.4 and 4.4]{CG2023quasi}, see also \cite{Charly-Toft-quasi2024}.
	\begin{theorem}\label{thmquasi}
	Let $s\geq0$ and $0<q\leq1$. Let $\sigma(x,D),\tau(x,D)$ be pseudodifferential operators with symbols $\sigma,\tau\in M^{\infty,q}_{1\otimes v_s}(\rdd)$. Then,\\
		$(i)$ $\sigma(x,D)\in B(L^2(\rd))$.\\
		$(ii)$ $\sigma(x,D)\tau(x,D)=\rho(x,D)$, with $\rho\in M^{\infty,q}_{1\otimes v_s}(\rdd)$.\\
		$(iii)$ If $\sigma(x,D)$ is invertible in $B(L^2(\rd))$, then $\sigma(x,D)^{-1}=\tau(x,D)$ for some $\tau\in M^{\infty,q}_{1\otimes v_s}(\rdd)$.
	\end{theorem}
For related  works in the framework of localization operators, see the recent developments in \cite{LUEF2018288,LUEF2019,LUEF2019-2}. 
\section{The Wigner kernel approach}\label{sec:3}

In this section we clarify the definition of the Wigner kernel and study the connection with the Schwartz kernel.
\begin{definition}\label{defWK}
	Let $T:\cS(\rd)\to\cS'(\rd)$ be a continuous linear operator. The Wigner kernel of $T$ is the distribution $k\in\cS'(\bR^{4d})$ such that 
	\begin{equation}\label{Wkernel}
		\la k,W(u,v)\otimes W(f,g)\ra=\la W(Tf,Tg),W(u,v)\ra, \qquad f,g,u,v\in\cS(\rd).
	\end{equation}
\end{definition}

\begin{remark}\label{rmrk32}
Assume that a Wigner kernel $k$ exists for the continuous linear operator $T$. Choosing $f=g$ in \eqref{Wkernel}, we retrieve \eqref{WignersWkernelIntegral}. In fact, interpreting the following integrals in the distributional sense,
\begin{align*}
	\la W(Tf),W(u,v)\ra & = \la k, W(u,v)\otimes Wf\ra\\
	& = \int_{\bR^{4d}}k(z,w)\overline{W(u,v)(z)}Wf(w)dzdw\\
	&=\int_{\rdd}\Big(\int_{\rdd}k(z,w)Wf(w)dw\Big)\overline{W(u,v)(z)}dz\\
	&=\Big\la \int_{\rdd}k(z,w)Wf(w)dw,W(u,v)\Big\ra.
\end{align*}
The density of $\mbox{span}\{W(u,v):u,v\in\cS(\rd)\}$ in $\cS(\rdd)$ entails formula \eqref{WignersWkernelIntegral}.
\end{remark}

\begin{theorem}\label{thmWkernel}
	Let $T:\cS(\rd)\to\cS'(\rd)$ be linear and bounded. Let $k_T\in\cS'(\rdd)$ be its Schwartz kernel, defined as in \eqref{Skernel}. Then, 
	\begin{equation}\label{kkT}
	k(x,\xi,y,\eta)=Wk_T(x,y,\xi,-\eta)
	\end{equation}
	is the unique tempered distribution satisfying \eqref{Wkernel}. Consequently, every continuous linear operator admits a unique Wigner kernel.
\end{theorem}

It is proved in \cite[Corollary 3.4]{CRGFIO1} that $k\in L^2(\bR^{4d})$ (resp. $\cS(\bR^{4d})$) if and only if $k_T\in L^2(\rdd)$ (resp. $\cS(\rdd)$). We show that a similar result holds in the wider context of modulation spaces.

\begin{proposition}\label{prop34}
	Let $T:\cS(\rd)\to\cS'(\rd)$ be a continuous linear operator, with Schwartz kernel $k_T$  and Wigner kernel $k$.	Consider $0<p,q\leq\infty$, $s\geq0$, and let either  $m=v_s$ or $m=(1\otimes v_s)$ (defined  either on $\rdd$ or on $\bR^{4d}$).  Then we have the following characterization:
	$$k\in M^p_{m}(\bR^{4d}) \Leftrightarrow k_T\in M^p_{m}(\rdd),$$ with
	\begin{equation}\label{th23bis}
	\norm{k}_{M^p_{m}}\asymp\norm{k_T}^2_{M^p_{m}}. 
	\end{equation}
\end{proposition}
\begin{proof}	
	By \eqref{kkT}, $k\in M^{p}_{ m}(\bR^{4d})$ if and only if $\mathfrak{T}_pWk_T\in M^{p}_{ m}(\bR^{4d})$, where $\mathfrak{T}_p$ is the change of variables in \eqref{kkT}. By item $(i)$ of Proposition \ref{prop25}, $\mathfrak{T}_pWk_T\in M^{p}_{ m}(\bR^{4d})$ if and only if $Wk_T\in M^{p}_{ m}(\bR^{4d})$. Hence,
	\[
		\norm{k}_{M^p_{m}}\asymp\norm{Wk_T}_{M^p_{m}}.
	\]

	If $0<p\leq2$, the assertion follows from \cite[Theorem 2.16]{CRcharModSp}, choosing $\cA=A_{1/2}$, as defined in \eqref{defA12}. We may assume $p>2$. Let $\Phi\in\cS(\bR^{4d})$. We have:
	\begin{align*}
		\norm{k}_{M^{p}_{ m}}\asymp\norm{Wk_T}_{M^{p}_{ m}}\asymp\norm{W(Wk_T, \Phi)}_{L^{p}_{ m}}=\norm{\hat\cA(k_T\otimes \bar k_T\otimes \Phi)}_{L^{p}_{ m}}=\norm{W_\cA(k_T\otimes\bar k_T,\Phi)}_{L^p_{m}},
	\end{align*}
	where $\hat\cA$ is the metaplectic operator associated to the matrix 
	\[
		\cA = \left(\begin{array}{c c | c c | c c | c c }
			I_{d\times d}/4 & I_{d\times d}/4 & I_{d\times d}/2 & 0_{d\times d} & 0_{d\times d} & 0_{d\times d} & 0_{d\times d} & 0_{d\times d}\\
			0_{d\times d} & 0_{d\times d} & 0_{d\times d} & I_{d\times d}/2 & I_{d\times d}/4 & -I_{d\times d}/4 & 0_{d\times d} & 0_{d\times d}\\
			\hline
			0_{d\times d} & 0_{d\times d} & 0_{d\times d} & 0_{d\times d} & I_{d\times d}/2 & I_{d\times d}/2 & -I_{d\times d}/2 & 0_{d\times d}\\
			-I_{d\times d}/2 & I_{d\times d}/2 & 0_{d\times d} & 0_{d\times d} & 0_{d\times d} & 0_{d\times d} & 0_{d\times d} & -I_{d\times d}/2\\
			\hline
			0_{d\times d} & 0_{d\times d} & 0_{d\times d} & 0_{d\times d} & I_{d\times d} & I_{d\times d} & I_{d\times d} & 0_{d\times d}\\
			-I_{d\times d} & I_{d\times d} & 0_{d\times d} & 0_{d\times d} & 0_{d\times d} & 0_{d\times d} & 0_{d\times d} & I_{d\times d}\\
			\hline
			-I_{d\times d}/2 & -I_{d\times d}/2 & I_{d\times d} & 0_{d\times d} & 0_{d\times d} & 0_{d\times d} & 0_{d\times d} & 0_{d\times d}\\
			0_{d\times d} & 0_{d\times d} & 0_{d\times d} & I_{d\times d} & -I_{d\times d}/2 & I_{d\times d}/2 & 0_{d\times d} & 0_{d\times d},
		\end{array}\right),
	\]
	where $\cA$ can be easily computed using \eqref{homomorph} and Theorem \ref{lift}. $W_\cA$ is shift-invertible with 
	\[
		E_\cA = \frac{1}{2}A_{1/2},
	\]
	where $A_{1/2}$ is defined as in \eqref{defA12}. By Theorem \ref{thm26} and Corollary \ref{cor29},
	\begin{align*}
	\norm{k}_{M^p_{m}}&\asymp	\norm{W_\cA(k_T\otimes \bar k_T, \Phi)}_{L^{p}_{ m}}\asymp\norm{k_T\otimes \bar k_T}_{M^{p}_{ m}}\leq\norm{k_T}^2_{M^p_{ m
	}}.
	\end{align*}
	On the other hand, for every $g\in \cS(\bR^{2d})$,
	\begin{align*}
		|\la k_T,g\ra|^2&=\la k_T,g\ra\overline{\la k_T,g\ra}=\la k_T\otimes \bar k_T,g\otimes\bar g\ra=\la Wk_T,Wg\ra=\la k,Wg\ra.
	\end{align*}
	Now, for  $k_I\in M^p_m$, using Lemma \ref{lemmaE0} (observe that $ p\geq 2$)
	\begin{align*}
		\norm{k_T}_{M^p_{m}}^2&=\Big(\sup_{g\in \cS(\rdd),\, \|g\|_{M^{p'}_{1/m}}=1}|\la k_T,g\ra|\Big)^2=\sup_{g\in \cS(\rdd),\,  \|g\|_{M^{p'}_{1/m}}=1}\la k,Wg\ra\\
		&\lesssim\sup_{G\in \cS(\bR^{4d}), \, \|G\|_{M^{p'}_{1/m}}=1}|\la k,G\ra|=\norm{k}_{M^p_{m}}.
	\end{align*}
	Consequently the estimate \eqref{th23bis} follows.
\end{proof}

For future works, we rewrite the above result for the special cases of Sobolev spaces $H^s$ (see, e.g., \cite{triebel2010theory}) and Shubin-Sobolev spaces $\Qs$ \cite{shubin}. 
\begin{corollary}
	Under the assumptions of the previous proposition,
	\begin{equation}\label{Sobolev-Shubin}
		\norm{k}_{\Qs}\asymp\norm{k_T}^2_{\Qs}\quad\mbox\quad \norm{k}_{H^s}\asymp\norm{k_T}^2_{H^s}. 
	\end{equation}
	\end{corollary}
\begin{proof}
	The equivalences follow from \eqref{th23bis} by choosing $p=2$ and either $m=v_s$ for the left-hand side of \eqref{Sobolev-Shubin} or $m=1\otimes v_s$ for the right-hand one.
\end{proof}

\section{The class $FIO(S,M^{\infty,q}_{1\otimes v_s})$}\label{sec:4}
In this section we shall study  the above-mentioned class.
	The first question we address is which metaplectic Wigner quantization can replace the Rihacek distribution in Definition \ref{defFIOclass}.
	
	Let us recall that upper block triangular symplectic matrices play a special role in the analysis on modulation spaces, since the associated metaplectic operators preserve $M^{p,q}_{m}$ for every $p,q$, as outlined in Proposition \ref{prop25}. 
	
	\begin{remark}\label{rem41}
		A matrix $U\in\mbox{Sp}(d,\bR)$ is upper block triangular if and only if $U=\mathcal{D}_LV_C^T$ for some $L\in\mbox{GL}(d,\bR)$ and $C\in\mbox{Sp}(d,\bR)$, where $\mathcal{D}_L$ and $V_C$ are defined as in \eqref{VCDL}. In particular, $\hat U\in \mbox{Mp}(d,\bR)$ has upper block triangular projection $U$ if and only if $\hat U=\mathfrak{T}_L\psi_C$, up to a sign, for some $L\in\mbox{GL}(d,\bR)$ and $C\in\mbox{Sp}(d,\bR)$, where $\mathfrak{T}_L$ and $\psi_C$ are defined as in Example \ref{ex23}.
	\end{remark}
	
	
	\begin{theorem}\label{thm41}
		Consider $h\in\cS'(\bR^{4d})$ and $W_\cA$  a metaplectic Wigner distribution satisfying
		 \begin{equation}\label{admissibleQuantiz}
		 	W_\cA(f,g)(z)=|\det L|^{1/2}\cF^{-1}\Phi_{-C}\ast W_0(f,g)(Lz), \qquad f,g\in\cS'(\rd),
		\end{equation}
		for some $C\in\mbox{Sym}(2d,\bR)$ and $L\in\mbox{GL}(2d,\bR)$. For $0<q\leq1$ and $s\geq0$, the following statements are equivalent: \\
		(i) $h$ is the Schwartz kernel of the Kohn-Nirenberg quantization of a pseudodifferential operator $\sigma(z,D)$, with $\sigma\in M^{\infty,q}_{1\otimes v_s}(\bR^{4d})$.\\
		$(ii)$ $h$ is the Schwartz kernel of the metaplectic pseudodifferential operator $Op_\cA(\sigma_\cA)$, with $\sigma_\cA\in M^{\infty,q}_{1\otimes v_s}(\bR^{4d})$.
	\end{theorem}
	\begin{proof}
	Obviously, $\cA A_0^{-1}=U$, with $U$ an upper block triangular if and only if $\cA=UA_0$. 	By Remark \ref{rem41}, $U=\mathcal{D}_LV_C^T$, for suitable $L\in\mbox{GL}(2d,\bR)$ and $C\in\mbox{Sym}(2d,\bR)$. 
	
	In terms of metaplectic Wigner distributions, \eqref{homomorph} entails that
	\[
		W_\cA(f,g)=\hat\cA(f\otimes \bar g)=\hat U\hat A_0 (f\otimes\bar g) =\mathfrak{T}_L\psi_C(W_0(f,g)), \qquad f,g\in \cS'(\rd),
	\]
	up to a sign, where $\psi_C$ and $\mathfrak{T}_L$ are defined as in Example \ref{ex23}. 
	
	Assume that $W_\cA$ satifies \eqref{admissibleQuantiz}. By Proposition \ref{prop25}, 
	\begin{equation}\label{relImp}
	\sigma_\cA=\hat\cA \hat A_0^{-1}\sigma = \hat U\sigma.
	\end{equation}
	Since, $\hat U\sigma\in M^{\infty,q}_{1\otimes v_s}(\bR^{4d})$ if and only if $\sigma\in M^{\infty,q}_{1\otimes v_s}(\bR^{4d})$, the assertion follows. 
	\end{proof}
	
	\begin{remark}
		In terms of matrix decompositions, $W_\cA$ satisfies \eqref{admissibleQuantiz} if and only the block decomposition \eqref{blocksA} of the projection $\cA$ satisfies:
		\begin{equation}\label{equivAdmissible}
			\begin{cases}
				A_{32}=-A_{31},\\
				A_{42} = -A_{41},\\
				A_{34}=A_{33},\\
				A_{44}=A_{43}.
			\end{cases}
		\end{equation}
		This follows by imposing the product $\cA A_0^{-1}$ to be upper triangular. 
	\end{remark}
	
	The matrix $A_{1/2}$, as defined in \eqref{defA12} satisfies \eqref{equivAdmissible}. This  entails the following corollary.
	
	\begin{corollary}\label{cor44}
	 Let $T:\cS(\rd)\to\cS'(\rd)$ be linear and continuous, 	 $0<q\leq1$, and $s\geq0$. The following are equivalent:\\
		(i) $T\in FIO(S,M^{\infty,q}_{1\otimes v_s})$.\\
		(ii) The Wigner kernel $k$ of $T$ satisfies $k(z,w)=h(z,Sw)$, where $h$ is the Schwartz kernel of a pseudodifferential operator $a^w(z,D)$ with $a\in M^{\infty,q}_{1\otimes v_s}(\bR^{4d})$.
	\end{corollary}
	
	The following issue proves Theorem \ref{tc1old} $(i)$.
	\begin{theorem}
	 If $T\in FIO(S,M^{\infty,q}_{1\otimes v_s})$, $0<q\leq1$, $s\geq0$, and $S\in\mbox{Sp}(d,\bR)$, then $T\in B(L^2(\rd))$.
	\end{theorem}
	\begin{proof}
		Let $k$ be the Wigner kernel of $T$ and $K$  the operator with Schwartz kernel $k(z,w)=h(z,Sw)$. Since $T\in FIO(S,M^{\infty,q}_{1\otimes v_s})$, $h$ is the kernel of a pseudodifferential operator $\sigma(z,D)$ with symbol in $\sigma\in M^{\infty,q}_{1\otimes v_s}(\bR^{4d})$. By item $(i)$ of Theorem \ref{thmquasi}, $a^w(z,D)$ is bounded on $L^2(\rd)$. This implies that $K$ is bounded on $L^2(\rd)$ as well. In fact, if $H$ denotes the operator with kernel $h$ and $\mathfrak{T}_Lh(z,w):=h(z,Sw)$
		\begin{align*}
			\la Kf,g\ra &= \la k,g\otimes \bar f\ra			=\la \mathfrak{T}_Lh,g\otimes \bar f\ra
			=\la h,\mathfrak{T}_L^{-1}(g\otimes \bar f)\ra \\
			&= \la h,g\otimes\mathfrak{T}_{S^{-1}}\bar f\ra
			=\la h,g\otimes\overline{\mathfrak{T}_{S^{-1}}f}\ra
			=\la H\mathfrak{T}_{S^{-1}}f,g\ra,
		\end{align*}
		for every $f,g\in \cS(\rdd)$, so that 
		\[
			|\la Kf,g\ra|=|\la H\mathfrak{T}_{S^{-1}}f,g\ra|\leq\norm{H}_{op}\norm{\mathfrak{T}_{S^{-1}}f}_2\norm{g}_2=\norm{H}_{op}\norm{f}_2\norm{g}_2,
		\]
		for every $f,g\in \cS(\rdd)$, which gives $\norm{K}_{op}\leq \norm{H}_{op}$. The same argument, with the roles of $K$ and $H$ reversed, gives:
		\[
			\norm{K}_{op}=\norm{H}_{op}.
		\]
		For every $f\in\cS(\rd)$, by Moyal's identity \eqref{moyal} applied to $\cA=A_{1/2}$ and \eqref{rmrk32},
		\[
			\norm{Tf}_2^2=\norm{W(Tf)}_2=\norm{K(Wf)}_2\leq\norm{K}_{op}\norm{Wf}_2=\norm{K}_{op}\norm{f}_2^2.
		\]
	In conclusion, $\norm{T}_{op}\leq \norm{K}_{op}^{1/2}$.
	\end{proof}

	In what follows we showcase the algebra property (Theorem \ref{tc1old} $(ii)$).
	\begin{theorem}\label{thm46}
		Consider $0<q\leq1$, $s\geq0$, and $S_1,S_2\in\mbox{Sp}(d,\bR)$. If $T_1\in FIO(S_1,M^{\infty,q}_{1\otimes v_s})$ and $T_2\in FIO(S_2,M^{\infty,q}_{1\otimes v_s})$ then $T_1T_2\in FIO(S_1S_2,M^{\infty,q}_{1\otimes v_s})$.
	\end{theorem}
	\begin{proof}
	It follows the pattern of Theorem 4.2 in \cite{CRGFIO1}. Let $k_j$ be the Wigner kernel of $T_j$ ($j=1,2$). A simple computation shows that the Wigner kernel of $T_1T_2$ is
	\begin{equation}
		k(z,w)=\int_{\rdd}k_1(z,u)k_2(u,w)du.
	\end{equation}
	Writing $k_j(z,w)=h_j(z,S_jw)$ ($j=1,2$) and using the change of variables $S_1u=u'$, 
	\begin{align*}
		k(z,w)&=\int_{\rdd}h_1(z,S_1u)h_2(u,S_2w)du=
		\int_{\rdd}h_1(z,S_1u)\cF_2^{-1}\sigma_2(u,S_2w)du\\
		&=\int_{\rdd}h_1(z,u)\cF_2^{-1}\sigma_2(S_1^{-1}u,S_1^{-1}u-S_2w)du\\
		&=\int_{\rdd}h_1(z,u)\mathfrak{T}_{M}\cF_2^{-1}\sigma_2(u,u-S_1S_2w)du\\
		&=\int_{\rdd}h_1(z,u)\cF_2^{-1}\tilde\sigma_2(u,u-S_1S_2w)du
		=\int_{\rdd}h_1(z,u)\tilde h_2(u,S_1S_2w)du,
	\end{align*}
	where $\tilde\sigma_2(u,v)=\sigma_2(S_1^{-1}u,S_1^Tv)$ and $\tilde h_2(z,w)=\cF_2^{-1}\tilde\sigma_2(z,z-w)$. By item $(i)$ of Proposition \ref{prop25} the  symbol $\tilde\sigma_2$ is in $M^{\infty,q}_{1\otimes v_s}(\bR^{4d})$. Consequently,
	\[
		k(z,w)=\int_{\rdd}k_1(z,u)\tilde k_2(u,w)du, 
	\]
	where $\tilde k_2(z,w)=\tilde h_2(z,S_1S_2w)$. The distributions $h_1$ and $\tilde h_2$ are the Schwartz kernels of the pseudodifferential operators $\sigma_1(z,D)$ and $\tilde\sigma_2(z,D)$, with $\sigma_1,\tilde\sigma_2\in M^{\infty,q}_{1\otimes v_s}(\bR^{4d})$, whereas
	\[
		h(z,w):=\int_{\rdd}h_1(z,u)\tilde h_2(u,w)du
	\]
	gives the Schwartz kernel of their composition, by \eqref{Skercomp}. By item $(ii)$ of Theorem \ref{thmquasi} and Corollary \ref{cor44}, $h$ is the Schwarz kernel of a pseudodifferential operator $\sigma(z,D)$ with symbol $\sigma\in M^{\infty,q}_{1\otimes v_s}(\bR^{4d})$. Since $k(z,w)=h(z,S_1S_2w)$, the assertion is proved.
	\end{proof}

	If  $T\in FIO(S,M^{\infty,q}_{1\otimes v_s})$ its adjoint $T^*$ belongs to the class $FIO(S^{-1},M^{\infty,q}_{1\otimes v_s})$, as shown below.
	\begin{lemma}\label{lemmaTast}
If  $T\in FIO(S,M^{\infty,q}_{1\otimes v_s})$, $0<q\leq1$, $s\geq0$, $S\in\mbox{Sp}(d,\bR)$, then $T^{\ast}\in FIO(S^{-1},M^{\infty,q}_{1\otimes v_s})$.
	\end{lemma}
	\begin{proof}
		A direct computation shows that the Wigner kernel of $T^\ast$ is $\tilde k(z,w)=k(w,z)$ (see \cite[Theorem 4.3]{CRGFIO1}). We have:
		\begin{align*}
			\tilde k(z,w)&=k(w,z)=h(w,Sz)=\tilde h(z,S^{-1}w),
		\end{align*}
		where, if
		\[
			L=\begin{pmatrix}
				0_{2d\times 2d} & S^{-1}\\
				S^{-1} & 0_{2d\times 2d}
			\end{pmatrix},
		\]
		 $\tilde h(z,w)=h(S^{-1}w,S^{-1}z)=\mathfrak{T}_Lh(z,w)$. It remains to prove that $\tilde h$ is the kernel of a pseudodifferential operator $\tilde\sigma(z,D)$ with symbol $\tilde\sigma\in M^{\infty,q}_{1\otimes v_s}(\bR^{4d})$. We have:
		 \begin{align*}
		 	\tilde h(z,w)=\mathfrak{T}_Lh(z,w)=\mathfrak{T}_L\mathfrak{T}_{L_0}^{-1}\cF_2^{-1}\sigma(z,w)=\mathfrak{T}_{L_0}^{-1}\mathfrak{T}_{L'}\cF_2^{-1}\sigma(z,w),
		 \end{align*}
		 with 
		 \[
		 	L'=L_0^{-1}LL_0=\begin{pmatrix}S^{-1} & -S^{-1}\\
			0_{2d\times 2d} & -S^{-1}\end{pmatrix}
		 \]
		 which is upper block triangular. Consequently, by Proposition \ref{intert}:
		 \[
		 	\tilde h(z,w)=\mathfrak{T}_{L_0}^{-1}\cF_2^{-1}\hat U \sigma (z,w),
		 \]
		 with $U=\pi^{Mp}(\hat U)$ upper block triangular. By Proposition \ref{prop25}, $\tilde\sigma=\hat U\sigma\in M^{\infty,q}_{1\otimes v_s}(\bR^{4d})$, and the Wigner kernel $\tilde h$ of $T^\ast$ is the Schwartz kernel of the pseudodifferential operator $\tilde\sigma(z,D)$, with $\tilde\sigma\in M^{\infty,q}_{1\otimes v_s}(\bR^{4d})$. Since $\tilde k(z,w)=\tilde h(z,S^{-1}w)$, it follows that $T^\ast\in FIO(S^{-1},M^{\infty,q}_{1\otimes v_s})$.
	\end{proof}

	To prove the Wiener property we will need the following issue, proved for the class $FIO(S)$ in Lemma $4.3$ of  \cite{CRGFIO1}, which is still valid for our more general classes.
	\begin{lemma}\label{elena} Assume $T\in FIO(S,M^{\infty,q}_{1\otimes v_s})$ invertible on $L^2(\rd)$, then the operator $K$ in \eqref{WignersWkernel} is invertible on $\lrdd$ with inverse $K^{-1}$ satisfying
		\begin{equation}\label{I3inv}
			K^{-1} W(f,g)=W(T^{-1}f,T^{-1}g), \quad f,g\in\cS(\rd).
		\end{equation}
		Namely, $K^{-1}$ is the operator with integral kernel given by the Wigner kernel of $T^{-1}$. 
	\end{lemma}
	\begin{theorem}
		If  $T\in FIO(S,M^{\infty,q}_{1\otimes v_s})$, $0<q\leq1$, $s\geq0$, $S\in\mbox{Sp}(d,\bR)$, then $T^{-1}\in FIO(S^{-1},M^{\infty,q}_{1\otimes v_s})$.
	\end{theorem}
	\begin{proof}
		Let $T^\ast$ be the $L^2$ adjoint of $T$. By Lemma \ref{lemmaTast}, $T^\ast\in FIO(S^{-1},M^{\infty,q}_{1\otimes v_s})$. Since $T$ is invertible in $B(L^2(\rd))$, $T^\ast$ is invertible in $B(L^2(\rd))$ and the operator $P=T^\ast T$ is invertible in $B(L^2(\rd))$. Moreover, $P\in FIO(I_{2d\times 2d},M^{\infty,q}_{1\otimes v_s})$ by Theorem \ref{thm46}.  In other words, the operator $K=K_P$ related to $P$ in  \eqref{I3} is a pseudodifferential operator with symbol in $M^{\infty,q}_{1\otimes v_s}(\bR^{4d})$.  By \cite[Theorem 4.6]{CG2023quasi} $K_P^{-1}$  is a pseudodifferential operator with symbol  in $M^{\infty,q}_{1\otimes v_s}(\bR^{4d})$ which implies that the related operator $P^{-1}$ is in the class $FIO(I_{2d\times 2d},M^{\infty,q}_{1\otimes v_s})$.
		Finally, the algebra property of Theorem \ref{thm46} gives $T^{-1}=P^{-1}T^\ast\in FIO(S^{-1},M^{\infty,q}_{1\otimes v_s})$,  as desired.
		
	\end{proof}
	
	\begin{corollary}
		$FIO(\mbox{Sp}(d,\bR),M^{\infty,q}_{1\otimes v_s})=\bigcup_{S\in \mbox{Sp}(d,\bR)}FIO(S,M^{\infty,q}_{1\otimes v_s})$ is a Wiener subalgebra of $B(L^2(\rd))$.
	\end{corollary}

\section{FIOs of type I and Schr\"odinger propagators}\label{sec:5}
We start this section by  considering FIOs of type I of the type \eqref{FIO1}
with  quadratic phases  $\Phi$  in \eqref{QuadTamePhase}. If $\sigma\in M^{\infty,q}_{1\otimes v_s}(\rdd)$ for some $0<q\leq1$ and $s\geq0$, we shall prove that the FIO  $T_I$ is in the class $FIO(S,M^{\infty,q}_{1\otimes v_s})$, with $S$ being the symplectic matrix  in \eqref{decompS}.
\begin{theorem}
	Let $T_I$ be a FIO of type I as above with $\sigma\in\cS(\rdd)$. Then, the Wigner kernel of $T_I$ is given by
	\begin{equation}\label{defKi}
		k_I(x,\xi,y,\eta)=\cF_2\sigma_I(x,\eta,\xi-CA^{-1}x-A^{-T}\eta,y-A^{-1}x+A^{-1}B\eta),
	\end{equation}
	where $\cF_2$ is the partial Fourier transform defined in \eqref{FT2}, and
	\begin{equation}\label{defsigmai}
		\sigma_I(x,\eta,t,r):=\sigma(x+t/2,\eta+r/2)\overline{\sigma(x-t/2,\eta-r/2)}.
	\end{equation}
Besides, if $\sigma\in M^{p,q}_{1\otimes v_s}(\rdd)$, $0<p,q\leq\infty$ and $s\geq0$, then $\sigma_I\in M^{p,q}_{1\otimes v_s}(\bR^{4d})$. 
\end{theorem}
\begin{proof}
	Formula \eqref{defKi} follows from the proof of Theorem 5.10 in \cite{CRGFIO1}. It remains to show that $\sigma_I\in M^{p,q}_{1\otimes v_s}(\bR^{4d})$. Consider
	\[
		M=\begin{pmatrix}
			I_{d\times d} & 0_{d\times d} & I_{d\times d}/2 & 0_{d\times d}\\
			0_{d\times d} & I_{d\times d} & 0_{d\times d} & I_{d\times d}/2\\
			I_{d\times d} & 0_{d\times d} & -I_{d\times d}/2 & 0_{d\times d}\\
			0_{d\times d} & I_{d\times d} & 0_{d\times d} & -I_{d\times d}/2
		\end{pmatrix}.
	\]
	Since $\det M=1$, we have:
	\[
		\sigma_I(x,\xi,y,\eta)=\mathfrak{T}_M(\sigma\otimes\bar\sigma)(x,\xi,y,\eta).
	\]
The assertion will follow from Proposition \ref{prop25} if we prove that $\sigma\otimes\bar\sigma\in M^{p,q}_{1\otimes v_s}(\bR^{4d})$, which follows directly by Corollary \ref{cor29}, item $(i)$.
\end{proof}

We are ready to show that $T_I$ is in the class $FIO(S,M^{\infty,q}_{1\otimes v_s})$.

\begin{theorem}\label{5.2}
		Let $T_I$ be a FIO of type I as above with symbol $\sigma\in M^{\infty,q}_{1\otimes v_s}$. Then $T_I\in FIO(S,M^{\infty,q}_{1\otimes v_s})$.
\end{theorem}
\begin{proof}
	Formula \eqref{defKi} reads as:
	\begin{equation}
		k_I(x,\xi,y,\eta)=\mathfrak{T}_{R_I}\cF_2\sigma_I(x,\xi,S(y,\eta)-(x,\xi)),
	\end{equation}
	where $S$ is the canonical transformation associated to $\Phi$ in \eqref{decompS}, and
	\begin{equation}	
		R_I=\left(\begin{array}{c c | c c}
			I_{d\times d} & 0_{d\times d} & 0_{d\times d} & 0_{d\times d}\\
			-C^T & A^T & -C^T & A^T\\
			\hline
			0_{d\times d} & 0_{d\times d} & CA^{-1} & -I_{d\times d}\\
			0_{d\times d} & 0_{d\times d} & A^{-1} & 0_{d\times d}
			\end{array}
		\right),
	\end{equation}
	which is upper block triangular. Hence, up to a sign, $\mathfrak{T}_{R_I}\cF_2=\cF_2\hat U_I$, where $U_I$ is upper block triangular. Let $\tilde\sigma_I:=\hat U\sigma_I$, so that
	\[
		k_I(x,\xi,y,\eta)=\cF_2 \hat U_I\sigma_I(x,\xi,S(y,\eta)-(x,\xi))=\cF_2\tilde\sigma_I(x,\xi,S(y,\eta)-(x,\xi)).
	\]
	By item $(i)$ of Proposition \ref{prop25}, $\tilde\sigma_I\in M^{\infty,q}_{1\otimes v_s}(\bR^{4d})$. By \eqref{kerandsymb}, this entails that $k_I(z,w)=h_I(z,Sw-z)$, where $h_I$ is the Schwartz kernel of the pseudodifferential operator $\tilde\sigma_I(z,D)$, with symbol $\tilde\sigma_I\in M^{\infty,q}_{1\otimes v_s}(\bR^{4d})$. Equivalently, $T_I\in FIO(S,M^{\infty,q}_{1\otimes v_s})$.
\end{proof}

The $L^2$-adjoint of a FIO of type $I$ is a FIO of type II, written  formally as
\begin{equation}\label{FIOII}
	T_{II}f(x)=\int_{\rdd}e^{-2\pi i[\Phi(y,\xi)-x\xi]}\tau(y,\xi)f(y)dyd\xi, \qquad f\in\cS(\rd).
\end{equation}
It was shown in \cite{CG2023quasi} that if $T_I$ is the FIO of type I above with symbol in $M^{\infty,q}_{1\otimes v_s}(\rdd)$ then its adjoint $T_{II}=T_I^{\ast}$ is bounded on $\lrd$ and has symbol $\tau \in M^{\infty,q}_{1\otimes v_s}(\rdd)$.\par
As a consequence of Lemma \ref{lemmaTast} we obtain
\begin{corollary}
If  $T_I\in FIO(S,M^{\infty,q}_{1\otimes v_s})$, $0<q\leq1$, $s\geq0$, $S\in\mbox{Sp}(d,\bR)$, then $T_{II}\in FIO(S^{-1},M^{\infty,q}_{1\otimes v_s})$.
\end{corollary}
For boundedness and further properties of FIOs we refer to \cite{concetti-garello-toft,CGNRJMPA,CGNRJMP2014}.
\subsection{Application to Schr\"{o}dinger equations}

An application of the theory above is the time-frequency representation of solutions to Cauchy problems for  Schr\"{o}dinger equations.
Our goal will be to represent the Wigner kernel of the  Schr\"{o}dinger propagator $e^{itH}$,  solution to the Cauchy problem 
\begin{equation}\label{C6intro}
	\begin{cases} i \displaystyle\frac{\partial
			u}{\partial t} +a(x,D)u+\sigma(x,D) u=0,\\
		u(0,x)=u_0(x),
	\end{cases}
\end{equation}
with initial datum $u_0\in\cS(\rd)$. The Hamiltonian  $H=a(x,D)u+\sigma(x,D)$ is the sum of two pseudodifferential operators. The first $a(x,D)$ is  the quantization of a quadratic form
whereas $\sigma(x,D)$ is the Kohn-Nirenberg operator expressing the perturbation. Here we work with symbols $\sigma$ in $M^{\infty,q}_{1\otimes v_s}(\rdd)$, $0<q\leq 1$.

First,  consider the unperturbed  case $\sigma=0$, namely
\begin{equation}\label{C12}
	\begin{cases} i \displaystyle\frac{\partial
			u}{\partial t} +a(x,D) u=0\\
		u(0,x)=u_0(x).
	\end{cases}
\end{equation}
The solution is given by metaplectic operators
$u=\widehat{S}_tu_0$, $t\in\bR$, for a suitable symplectic matrix
$\mathcal{S}_t$, see  \cite[Chp. 4]{Gos11}. Precisely, if $a(x,\xi)=\frac{1}{2}\xi \mathbb{B}\xi+\xi \mathbb{A}
x-\frac{1}{2}x\mathbb{C}x$, with $\mathbb{B}, \mathbb{C}$ symmetric, we
can consider the classical evolution, given by the linear Hamiltonian
system
\begin{equation}\label{sistemE0}
\begin{cases}
	2\pi \dot x=\nabla  _\xi a  =\mathbb{A}x+\mathbb{B}\xi\\
	2\pi \dot \xi=-\nabla _x a =\mathbb{C} x-\mathbb{A}^T\xi
\end{cases}
\end{equation}
(the factor $2\pi$ depends on our normalization of the Fourier transform) with Hamiltonian matrix $\mathbb{S}:=\begin{pmatrix}\mathbb{A}&\mathbb{B}\\
	\mathbb{C}&-\mathbb{A}^T\end{pmatrix}\in \mathrm{sp} (d,\bR )$.

Where $\mathrm{sp} (d,\bR )$ is the Lie algebra  of $Sp(d,\bR)$ and consists of all $X\in M(2d,\bR)$
such that
$XJ + JX^T = 0 $.
 The solution to \eqref{sistemE0} is the linear symplectic diffeomorphism  $S_t:\bR_{y,\eta}^{2d}\to \bR_{x,\xi}^{2d}$, for every fixed $t\in\bR$, with
 $S_t(y,\eta)=(x(t,y,\eta), \xi(t,y,\eta))$. Hence, for every $t\in\bR$, 
 $S_t=e^{t\mathbb{S}}\in Sp(d,\R)$. We assume  $S_t$ having  block decomposition
	\begin{equation}\label{decompSt}
	S_t=\begin{pmatrix}
		A_t & B_t\\ C_t & D_t
	\end{pmatrix}, \qquad A_t,B_t,C_t,D_t\in\bR^{d\times d}.
\end{equation}

Since $$e^{t\mathbb{S}}=\sum_{k=0}^\infty \frac{t^k}{k!} \mathbb{S}^k =I+t\sum_{k=1}^\infty \frac{t^{k-1}}{(k-1)!} \mathbb{S}^k , $$
There exists a $T^*>0$ such that the block $A_t$ in formula \eqref{decompSt} is invertible for every $|t|<T^*$. Precisely,
\begin{equation}\label{tempoE0}
	\exists\, T^*>0\quad\mbox{such\, that }\, \det A_t\not=0\qquad \forall t:\, |t|<T^*.
\end{equation}
 
We now come back to the perturbed problem. 
The propagator $e^{itH}$ providing the solution $u(t,x)=e^{itH}u_0(x)$ to \eqref{C6intro} is formally given by
 a generalized metaplectic operator. The result was initially proved for perturbations $\sigma(x,D)$ with  $\sigma\in M^{\infty,1}_{1\otimes v_s}(\rdd )$ in \cite[Theorem 4.1]{CGNRJMP2014} and  extended  to  $\sigma\in M^{\infty,q}_{1\otimes v_s}(\rdd )$, $0<q\leq1$ in \cite[Theorem 7.2]{CGRPartII2022}.
Namely, 
\begin{theorem}\label{teofinal0}
	Let $H = a(x,D) + \sigma(x,D)$ be a Hamiltonian with a real quadratic
	homogeneous polynomial $a$ and  $\sigma\in
	M^{\infty,q}_{1\otimes v_s}(\rdd ) $, $0<q\leq1$. Then the corresponding
	propagator $ e^{itH}$ is a so-called  \emph{generalized metaplectic operator} for every	$t \in\bR$.
	Specifically, the solution of the homogenous problem  \eqref{C12} is given by a
	metaplectic operator $\widehat{S_t}$, and $e^{itH}$
	is of the from
	\begin{equation}\label{soluzioneE1}
	e^{itH} = \widehat{S_t}b_t(x,D),
	\end{equation}
	for some symbol $b_t\in M^{\infty,q}_{1\otimes v_s}(\rdd)$.
\end{theorem}

In  \cite{CGRPartII2022}, starting from \eqref{soluzioneE1} and the perturbation symbol $\sigma $ in the H\"ormander class
$ S^0_{0,0}(\rdd)$,  an expression of the Wigner kernel was obtained.
We present here an alternative approach based on the algebra property which allows to obtain the Wigner kernel for a more general class of perturbations. Recall that $S^0_{0,0}(\rdd)\subset M^{\infty,q}_{1\otimes v_s}(\rdd)$.

First, we shall show that the generalized metaplectic operator above can be represented as a FIO of type I, for every $t$ satisfying $|t|<T^*$, with $T^*$ in \eqref{tempoE0}.
\begin{proposition}\label{fio1}
Under the assumptions of Theorem \ref{teofinal0} there exists $T^*>0$ such that the solution $e^{itH}$ in \eqref{soluzioneE1}  can be written as a Type I FIO $T_I$ as in \eqref{FIO1}:
\begin{equation}\label{fio1E0}
	T_{I,t} f(x)=\intrd e^{2\pi i \Phi_t(x,\eta)} \sigma_t(x,\eta)\hat{f}(\eta)d\eta,
\end{equation}
with a quadratic phase $\Phi_t(x,\eta)$ as in \eqref{QuadTamePhase} and symbol $\sigma_t\in M^{\infty,q}_{1\otimes v_s}(\rdd)$.
\end{proposition}
 \begin{proof}
By formula \eqref{soluzioneE1}, we have 
$e^{itH} = \widehat{S_t}b_t(x,D)$
with $S_t$  in \eqref{decompSt}. Moreover, there exists $T^*>0$ such that condition \eqref{tempoE0} holds. For every $t$ with $|t|<T^*$, $\det A_t\not=0$ implies the  following
formula for the metaplectic
operator $\widehat{S_t}$ \cite[Theorems 4.51]{folland89}:
\begin{equation}\label{f4}
	\widehat{S_t}f(x)=(\det
	A_t)^{-1/2}\intrd e^{-2\pi i \Phi_t(x,\eta)}\hat{f}(\eta)\,d\eta,
\end{equation}
	where
	$$	\Phi_t(x,\eta)=\frac12 x\cdot C_tA_t^{-1}x+
	\eta\cdot A_t^{-1} x-\frac12
	\eta\cdot
	A_t^{-1}B_t\eta.$$
Assuming $b_t(x,D) =\intrd b_t(x,\eta) \hat{f}(\eta) d\eta$, with $b_t\in M^{\infty,q}_{1\otimes v_s}(\rdd)$, we obtain  
$$\widehat{S_t}b_t(x,D)f(x)=(\det
A_t)^{-1/2}\intrd  e^{-2\pi i \Phi_t(x,\eta)} b_t(x,\eta) \hat{f} (\eta) d\eta,$$
and the conclusion follows by taking $\sigma_t= (\det
A_t)^{-1/2}  b_t(x,\eta)$.
 \end{proof}

Even though the type I representation holds only for small $t$ in general, the Wigner kernel representation of the propagator remains valid for every $t\in\bR$, as proved below.

\begin{theorem}\label{teofinal}
	Under the assumptions above, the propagator $e^{itH}$ is in the class $FIO(S_t,M^{\infty,q}_{1\otimes v_s})$ for every $t \in\bR$, where $S_t\in Sp(d,\bR)$ is the  solution to the Hamiltonian system  \eqref{sistemE0} written in \eqref{decompSt}.
	Namely, the Wigner kernel of $e^{itH}$ is given by 
	\[
	k(t,z,w) =h_t(z,S_tw)
	\]
	where $h_t$ is the kernel of a pseudodifferential operator of symbol  $b_t\in  M^{\infty,q}_{1\otimes v_s}(\bR^{4d})$, with continuous dependence of $t\in\bR$.
\end{theorem}
\begin{proof}
	By Proposition \ref{fio1} there exists a $T^*>0$ such that  $e^{itH}$  can be represented as a type I FIO as in \eqref{fio1E0}. By Theorem \ref{5.2}, we infer $$e^{itH}\in FIO(S_t,M^{\infty,q}_{1\otimes v_s}),\quad t\in (-T^*,T^*),$$
	where $S_t\in Sp(d,\bR)$ is the canonical transformation in \eqref{decompSt} satisfying $\det A_t\not=0$. In other words, the Wigner kernel $k(t,z,w)$  of $e^{itH}$  satisfies the equality 
	\begin{equation}\label{nucleoFIOt}
		k(t,z,w) =h_t(z, S_t(w)),\quad z,w\in \rdd,
	\end{equation}
	 for $t\in  (-T^*,T^*)$.
	
To extend the Wigner kernel property to every $t\in\bR$  we use the algebra property of $FIO(S_t,M^{\infty,q}_{1\otimes v_s})$ and the classical trick  reported below \cite{CRGFIO1}.
	
	Consider $T_0<T^*/2$ and define $I_h=(hT_0,(h+2)T_0)$, $h\in\bZ$. For $t\in\bR$, there exists an $h\in\bZ$ such that $t \in I_h$. For $t_1=t-hT_0\in (-T^*,T^*)$ we have that $e^{it_1H}\in FIO(S_{t_1})$ and for $t_2=\frac{h}{|h|}T_0\in (-T^*,T^*)$, $e^{it_2H}\in FIO(S_{t_2})$.
	Using the semigroup property of $e^{itH}$ and the  algebra property in Theorem \ref{tc1old} $(ii)$ we obtain
	$$e^{it_1H}(e^{it_2H})^{|h|}\in FIO(S_{t_1}S_{t_2}^{|h|},M^{\infty,q}_{1\otimes v_s}).$$
	By the group property of $S_t$:
	$$S_{t_1}S_{t_2}^{|h|}=S_{t_1+|h|t_2}=S_t,$$
	so that
	\begin{equation}\label{meta}
		e^{itH}\in FIO(S_t,M^{\infty,q}_{1\otimes v_s}),\quad t\in \bR,
	\end{equation}
	and the kernel representation in \eqref{nucleoFIOt} holds for every $t\in\bR$.
\end{proof}

\begin{remark}
(i)	Let us underline that this representation  holds for every $t\in\bR$, also in the so-called \emph{caustic points}, where the type I representation in \eqref{fio1E0} fails.\\
(ii) Theorem \ref{teofinal} is an extension of Theorem 6.1 in \cite{CRGFIO1}, since the H\"ormander class $S^0_{0,0}(\rdd)$  is obtained by intersection of modulation spaces:
$$S^0_{0,0}(\rdd)=\bigcap_{s\geq 0} M^{\infty,q}_{1\otimes v_s}(\rdd).$$
\end{remark}
Future developments of this theory could involve metaplectic Wigner distributions. Precisely, a new field of investigation would consist in  replacing  the Wigner kernel by a more general kernel, obtained applying   a shift-invertible metaplectic Wigner $W_\cA$ in \eqref{I3}, \eqref{I4}. We conjecture that similar results should be true.  
\section*{Acknowledgements}
The first two authors have been supported by the Gruppo Nazionale per l’Analisi Matematica, la Probabilità e le loro Applicazioni (GNAMPA) of the Istituto Nazionale di Alta Matematica (INdAM).


\begin{thebibliography}{}
	

%
%
%
%
\bibitem{CSW} M. Cappiello, R. Schulz, and P. Wahlberg.
Lagrangian Distributions and Fourier Integral
Operators with Quadratic Phase Functions and
Shubin Amplitudes. \textit{Publ. RIMS Kyoto Univ.} 56:561--602, 2020.
	\bibitem{Cohen2} L. Cohen. Time Frequency Analysis: Theory and Applications,
	Prentice Hall, 1995.
	\bibitem{CG2} L. Cohen and L. Galleani. Nonlinear transformation of differential equations into phase space, \textit{EURASIP J. Appl. Signal Process.} 12:1770–-1777, 2004.
%
	\bibitem{concetti-garello-toft} F. Concetti, G. Garello and J. Toft. Trace ideals for Fourier integral operators with non-smooth symbols II. {\em Osaka J. Math.}, 47(3): 739--786, 2010.
%
%
%

\bibitem{AEFN06}
E.~Cordero, F.~De~Mari, K.~Nowak, and A.~Tabacco.
\newblock Analytic features of reproducing groups for the metaplectic
representation.
\newblock {\em JFAA}, 12(3):157--180, 2006.
\bibitem{CG2023quasi}
E. Cordero, G. Giacchi. Quasi-Banach algebras and Wiener properties for pseudodifferential and generalized metaplectic operators. {\em Journal of Pseudo-Differential Operators and Applications}, 14(1), 9, 2023.

\bibitem{CG2023}
	E. Cordero and G. Giacchi. Symplectic analysis of time-frequency spaces. {\em Journal de Mathématiques Pures et Appliquées}, 177, 154-177, 2023.

	\bibitem{CGRPartII2022}
E.~Cordero, G. Giacchi and L. Rodino.
Wigner Analysis of Operators. Part II: Schr\"{o}dinger equations. {\em Commun. Math. Phys.}, to appear. arXiv:2208.00505
\bibitem{CRGFIO1} E.~Cordero, G. Giacchi and L. Rodino. Wigner Representation of  Schr\"odinger Propagators. {\em Submitted}. arXiv:2311.18383v2 

%
	\bibitem{CGNRJMPA} E.~Cordero, K.~Gr\"ochenig, F. Nicola and L. Rodino. Wiener algebras of Fourier integral operators. {\em J. Math. Pures Appl. (9)}, 99(2):219--233, 2013
\bibitem{CGNRJMP2014} E.~Cordero, K.~Gr\"ochenig, F. Nicola and L. Rodino. Generalized Metaplectic Operators and the Schr\"odinger   Equation  with  a  Potential in the  Sj\"ostrand Class, {\em J. Math. Phys.}, 55(8):081506, 17, 2014
\bibitem{Elena-book} E. Cordero and L. Rodino, \emph{ Time-Frequency Analysis of Operators}, De Gruyter Studies in Mathematics, 2020.
	\bibitem{CRPartI2022}
	E. Cordero and L. Rodino. {\it Wigner Analysis of Operators. Part I: Pseudodifferential Operators and Wave Front Sets}. Appl. Comput. Harmon. Anal. 58, 85-123, 2022.
	\bibitem{CRcharModSp}
	E.~Cordero and L. Rodino. {Characterization of Modulation Spaces By Symplectic Representations and Applications to Schr\"odinger Equations}. {\em Journal of Functional Analysis}, 284(9), 109892, 2023.
%
%
%
%
%

	\bibitem{F1}  H.~G.~Feichtinger,
	\newblock Modulation spaces on locally
	compact abelian groups,
	\newblock {\em Technical Report, University Vienna, 1983,} and also in
	\newblock {\em Wavelets and Their Applications},
	M. Krishna, R. Radha,  S. Thangavelu, editors,
	\newblock Allied Publishers,  99--140,  2003.
%
%
%
	\bibitem{folland89}
	G.~B. Folland.
	\newblock {\em Harmonic analysis in phase space}.
	\newblock Princeton Univ. Press, Princeton, NJ, 1989.

\bibitem{Fuhr}
	H. F\"uhr \& I. Shafkulovska. The metaplectic action on modulation spaces. {\em Applied and Computational Harmonic Analysis}, 68, 101604, 2024.

\bibitem{CG1}	L. Galleani, L. Cohen, The Wigner distribution for classical systems, Phys. Lett. A, 302(4):149--155, 2002.
%
	\bibitem{Galperin2004}
	Y.~V. Galperin and S.~Samarah.
	\newblock Time-frequency analysis on modulation spaces {$M^{p,q}_m$}, {$0<p,\
		q\leq\infty$}.
	\newblock {\em Appl. Comput. Harmon. Anal.}, 16(1):1--18, 2004.
	\bibitem{Gos11}
M.~A. de~Gosson.
\newblock {\em Symplectic methods in harmonic analysis and in mathematical
	physics}, volume~7 of {\em Pseudo-Differential Operators. Theory and
	Applications}.
\newblock Birkh\"auser/Springer Basel AG, Basel, 2011.%
%
\bibitem{Charly-Toft-quasi2024}
	K.~Gr{\"o}chenig, C. Pfeuffer and J. Toft. 
	\newblock Spectral invariance of quasi-Banach algebras of matrices and pseudodifferential operators. {\it Forum Mathematicum}, 2024. doi:10.1515/forum-2023-0212
\bibitem{CharlyIrina2024}
	K.~Gr{\"o}chenig and I. Shafkulovska.
\newblock	Benedicks-type uncertainty principle for metaplectic time-frequency representations. {\it Preprint}. arXiv:2405.12112
%
%
%
%
\bibitem{LUEF2018288}
F. Luef and E. Skrettingland.
\newblock{Convolutions for localization operators}.
\newblock{\em Journal de Mathématiques Pures et Appliquées}, 118:288--316, 2018.

\bibitem{LUEF2019}
F. Luef and E. Skrettingland.  
\newblock{Mixed-state localization operators: Cohen's class and
trace class operators}. \newblock{\em Journal of Fourier Analysis and Applications}, 25(4):2064–-2108, 2019. 
\bibitem{LUEF2019-2} F. Luef and E. Skrettingland, \newblock{On accumulated Cohen's class distributions and mixed-state
localization operators}.  Constr. Approx., 52(1): 31–-64, 2019.
%
%
	\bibitem{Kirkwood}	J.G. Kirkwood. Quantum statistics of almost classical assemblies, \textit{ Phys. Rev.}, 44:31--37, 1933.
		\bibitem{Kob2005} M. Kobayashi. Modulation spaces $M^{p,q}$,  for $0<p,q\leq\infty$.  \textit{Journal of Function Spaces and Applications}, 4(3):329--341, 2006. 
		\bibitem{MO} C. Mele and Alessandro Oliaro.
		Regularity of global solutions of partial differential equations in non isotropic ultradifferentiable spaces via time-frequency methods.
		 \textit{Journal of Differential Equations}, 286:821--855,	2021.
		\bibitem{Moyal} J.E. Moyal. Quantum mechanics as a statistical theory. \textit{Mathematical Proceedings of the Cambridge Philosophical Society}, 45(1):99--124, 1949. 
	\bibitem{Rudin}
	W. Rudin. \textit{Functional Analysis }(second edition), McGraw Hill Education, 1991. 
%
%
%
%
	\bibitem{shubin} M.A. Shubin. {\em Pseudodifferential operators and spectral theory}. Springer-Verlag, Berlin, second edition, 2001.
%
	\bibitem{A76}   J. Sj\"ostrand. An algebra of pseudodifferential operators. {\em Math. Res. Lett.}, 1(2):185--192, 1994.
%
%
%
%
%
%
%
%
%
%
\bibitem{triebel2010theory}
H. Triebel.
\newblock\emph{Theory of Function Spaces},
Modern Birkh{\"a}user Classics, Springer Basel, 2010.

	\bibitem{Wigner}
	E. Wigner.  On the Quantum Correction for Thermodynamic Equilibrium. {\it Phys. Rev.}, 40(5):749-759, 1932.
%
%
%
%
\end{thebibliography}
\end{document}